%% file: DLA_wedge_combined_0406.tex
\documentclass[12pt, one sided]{amsart}
\RequirePackage[colorlinks,citecolor=blue,urlcolor=blue]{hyperref}

\usepackage{latexsym,amssymb,amsmath,amsfonts,amsthm}
\usepackage{amsthm,amsmath}
\usepackage{float}
\usepackage{enumerate}
\usepackage[margin=3cm]{geometry}
\usepackage{bbm}
\usepackage{subfigure}
\usepackage{graphicx}
\usepackage[toc,page]{appendix}
\usepackage{multirow}
\usepackage{amsfonts}
\usepackage[all]{xy}
\usepackage{caption}
 \usepackage{geometry}                
\usepackage{graphicx}
\usepackage{amssymb}
\usepackage{epstopdf}
\usepackage{color}
\usepackage{bbm, dsfont}
\usepackage{hyperref}
\usepackage[utf8]{inputenc}
\usepackage{amssymb,amsthm,amsmath,amssymb,wrapfig,dsfont}
\usepackage[dvipsnames]{xcolor}
\definecolor{myred}{RGB}{251,154,133}
\definecolor{myblue}{RGB}{153,206,227}
\definecolor{mylightblue}{RGB}{0, 150, 255}
\definecolor{mygreen}{RGB}{32, 210, 64}
\definecolor{mygray}{RGB}{220, 220, 220}

\usepackage{tikz}
\usetikzlibrary{decorations.pathmorphing}
\tikzset{snake it/.style={decorate, decoration=snake}}
\usetikzlibrary{shapes.geometric,positioning,decorations.pathreplacing}

\newtheorem{theorem}{Theorem}

\newtheorem{lemma}{Lemma}[section]
\newtheorem{remark}{Remark}[section]
\newtheorem{Proposition}{Proposition}[section]
\newtheorem{corollary}{Corollary}[section]
\newtheorem{conjecture}{Conjecture}

\def\beq{ \begin{equation} }
\def\eeq{ \end{equation} }

\def\ep{\varepsilon}

\def\square{\vcenter{\vbox{\hrule height .4pt
  \hbox{\vrule width .4pt height 5pt \kern 5pt
        \vrule width .4pt} \hrule height .4pt}}}

\def\RR{\mathbb{R}}
\def\ZZ{\mathbb{Z}}

\newcommand{\BC}{{\mathbb{C}}}

\newcommand{\BN}{{\mathbb{N}}}

\newcommand{\BP}{{\mathbb{P}}}

\newcommand{\BR}{{\mathbb{R}}}

\newcommand{\BZ}{{\mathbb{Z}}}

\newcommand{\CD}{{\mathcal{D}}}

\newcommand{\CH}{{\mathcal{H}}}

\newcommand{\CR}{{\mathcal{R}}}

\newcommand{\CT}{{\mathcal{T}}}

\newcommand{\CW}{{\mathcal{W}}}

\newcommand{\tbE}{{\textbf{E}}}

\newcommand{\tbP}{{\textbf{P}}}

\newcommand{\bae}{\begin{equation}\begin{aligned}}
\newcommand{\eae}{\end{aligned}\end{equation}}
\newcommand{\ev}{\mathbf{E}}

\newcommand{\si}{\sigma}

\begin{document}

\title{Stabilization of DLA in a wedge}

\author{Eviatar B. Procaccia}
\thanks{The first author was supported by NSF grant DMS-1407558}
\address[Eviatar B. Procaccia]{Department of mathematics, Texas A\&M University}
\urladdr{http://www.math.tamu.edu/~procaccia}
\email{procaccia@tamu.edu}

\author{Ron Rosenthal}
\thanks{The second author was supported by ISF grant 771/17}
\address[Ron Rosenthal]{Department of mathematics, Technion - I.I.T.}
\urladdr{http://ron-rosenthal.net.technion.ac.il}
\email{ron.ro@technion.ac.il}
 
\author{Yuan Zhang}
\address[Yuan Zhang]{Department of mathematics, Texas A\&M University}
\urladdr{http://www.math.tamu.edu/~yzhang1988/}
\email{yzhang1988@math.tamu.edu}

\maketitle
\begin{abstract}
We consider Diffusion Limited Aggregation (DLA) in a two-dimensional wedge. We prove that if the angle of the wedge is smaller than $\pi/4$, there is some $a>2$ such that almost surely, for all $R$ large enough, after time $R^a$ all new particles attached to the DLA will be at distance larger than $R$ from the origin. This means that DLA stabilizes in growing balls, thus allowing a definition of the infinite DLA in a wedge via a finite time process.   
\end{abstract}

\tableofcontents


\section{Introduction}
	{Diffusion Limited Aggregation} (DLA) was introduced in 1983 by E. Witten and L. M. Sander \cite{witten1983} in order to study the geometry and dynamics of physical aggregation systems governed by diffusive laws. On the Euclidean lattice $\BZ^2$, DLA is a random process $(A_n)_{n\geq 0}$ of growing subsets of $\mathbb{Z}^2$, which are defined recursively. Typically, one fixes  $A_0:=\{(0,0)\}$, and given $A_n$, defines $A_{n+1}:=A_n\cup\{a_{n+1}\}$, where $a_{n+1}$ is a point sampled according to the harmonic measure of $\partial A_n $ from infinity. More precisely, $a_{n+1}$ is the first hitting place of $\partial A_n$ (the outer boundary of $A_n$) by a simple random walk started from distance $R$, in the limit $R\to\infty$ (See Section \ref{sec:pre} for the precise definition).

In this paper we study DLA in a two-dimensional wedge 
\[
	W_{\theta_1, \theta_2}=\big\{(x,y)\in\BZ^2: \arctan (y/x)\in[\theta_1,\theta_2],\, x\geq 0\big\},
\]
where $-\pi/2\le \theta_1<\theta_2\le \pi/2$. Here we used the convention that $(0,0)$ belongs to all wedges and that $\arctan(y/0)$ equals $\pi/2$ for $y>0$ and $-\pi/2$ for $y<0$. In Appendix \ref{section: harmonic} we prove the existence of the harmonic measure in $W_{\theta_1, \theta_2}$ which is needed for the definition of the DLA in the wedge. 

For $R>0$, let $B_R=\{(x,y)\in\BZ^2 ~:~ x^2+y^2<R^2\}$ be the discrete Euclidean ball of radius $R$ around the origin,  and define $W_{\theta_1,\theta_2}^R= W_{\theta_1,\theta_2}\cap B_R$. Throughout the paper, we consider $W_{\theta_1,\theta_2}$ and $W^R_{\theta_1,\theta_2}$ for $R>0$ as graphs, with vertices $W_{\theta_1,\theta_2}$ and $W^R_{\theta_1,\theta_2}$ respectively and edges induced from the graph $\BZ^2$. We denote by $\mathbf{P}_{\theta_1,\theta_2}^x$ the law of a simple random walk $(S_n)_{n\geq 0}$ in the graph $W_{\theta_1,\theta_2}$, starting from $x$ and for $B\subset W_{\theta_1,\theta_2}$, denote by $\tau_B^+=\inf\{n\geq 1~:~S_n\in B\}$ the first return time of the random walk into the set $B$. Finally, we set $\BP=\BP_{\theta_1,\theta_2}$ to be the law of the DLA $(A_n)_{n\geq 0}$ in $W_{\theta_1,\theta_2}$ (see Section \ref{sec:pre} for the formal definition). For future use, for $n\geq 1$ we denote by $a_n$ the particle added to the aggregate at time $n$, namely the unique vertex in $W_{\theta_1,\theta_2}$ such that $A_n = A_{n-1}\cup \{a_n\}$.

Our main result is the stabilization of the DLA in sufficiently sharp wedges.
\begin{theorem}\label{thm:main}
	Assume $-\pi/2\le \theta_1<\theta_2\le \pi/2$ satisfy $\theta_2-\theta_1<\pi/4$ and fix $a>\frac{2\pi+4(\theta_2-\theta_1)}{\pi-4(\theta_2-\theta_1)}$. Then $\BP_{\theta_1,\theta_2}$-almost surely, for every $R>0$ sufficiently large, the random sets $(A_n\cap B_R)_{n\geq R^a}$ are all the same. In other words, for all $R$ sufficiently large, none of the particles $(a_n)_{n\geq R^a}$ added to the system after time $R^a$ will attach to the aggregate inside $W^R_{\theta_1,\theta_2}$. 
\end{theorem}

The main tool in proving Theorem \ref{thm:main} is a discrete Beurling estimate for random walk in a wedge, which enables us to control the harmonic measure of finite, connected subsets of $W_{\theta_1,\theta_2}$. Unlike in the work of H. Kesten \cite{kesten1987}, the proof of the discrete Beurling estimate here does not rely on Green function calculations. 

For $A\subset W_{\theta_1,\theta_2}$ denote by $\partial A=\{y\in W_{\theta_1,\theta_2}\setminus A ~:~ \exists x\in A \text{ such that }\|x-y\|_1=1\}$ the outer boundary of $A$.

\begin{theorem}\label{thm:Beu_estimate}
	Fix $-\pi/2\le \theta_1<\theta_2\leq \pi/2$. For every $\ep>0$, there exists $M\in\BN$ and $C\in (0,\infty)$ such that for every $r,L\in\BN$ satisfying $r\geq M$ and $L/r\geq M$, every $R>0$ sufficiently large (depending on $\ep$ and $L$), every connected subset $A\subset W_{\theta_1,\theta_2}$, such that $A\cap \partial W_{\theta_1,\theta_2}^L \neq \emptyset$ and every $x\in \partial W^R_{\theta_1,\theta_2}$
	\begin{equation}\label{eq:Beu_est_1}
	\mathbf{P}^x_{\theta_1, \theta_2}\left(\tau^+_{ \partial W^r_{\theta_1, \theta_2}}\le\tau^+_{\partial A}\right)\le  C\left(\frac{r}{L} \right)^{\frac{\pi}{2(\theta_2-\theta_1)}-\ep}r \log L
	\end{equation}
\end{theorem}


\section{Discussion and open problems}\label{sec:disc}
	Since the introduction of DLA in 1983, rigorous understanding of the model was limited. The main exception being H. Kesten's upper bound on the growth rate \cite{kesten1987}, see also \cite{BY17}. Lately, similar results were obtained for DLA in the upper half plane with Dirichlet boundary conditions \cite{procaccia2017harmonic,procaccia2017zeroharmonic}. The main technical difference between this paper and previous works is that the former does not use any Green function calculations. The main reason is the lack of control over the discrete Green function in the wedge that would allow hitting probability calculations (See \cite{ganguly2017convergence} for the best known control in the case of general Neumann boundary conditions). 

There are many interesting open questions regarding DLA. First natural questions are about the growth rate, the fractal dimension and the relation between the two (see \cite{halsey1986scaling}). For these important questions our paper does not add to the discussion. Another natural question is about the number of arms in DLA growing in a wedge (or in $\BZ^2$). The physics literature does not provide clear conjectures or even definitions for the number of arms in a wedge. In \cite{kessler1998diffusion}, D. A. Kessler, Z. Olami, J. Oz, I. Procaccia, E. Somfai and L. M. Sander claim evidence for a critical angle $\nu$ between 120 and 140 degrees which guarantees coexistence of two arms in a wedge of angle $\nu$. 

One immediate contribution of our result is to provide a method to sampling the DLA in $W^R_{\theta_1,\theta_2}$, for every finite $R>0$, via a finite time random process. By Theorem \ref{thm:main} there is some $a>0$ such that for any $R$ large enough almost surely the sets $(A_n\cap B_R)_{n\ge R^a}$ are all the same. As a result, for all $R>0$ sufficiently large, we can define the DLA in $W^R_{\theta_1,\theta_2}$ to be $A_{R^a}\cap B_R$, which is a finite time random process. 

Returning to discuss the number of arms, since the sets $(A_{R^a}\cap B_R)$ are monotonic increasing in $R$, we can define
\[
	A_\infty= \bigcup_{n=0}^\infty A_n = \lim_{R\rightarrow\infty}(A_{R^a}\cap B_R).
\]

Let $\gimel$ be an infinite graph. The number of ends of $\gimel$ is defined to be the supremum on the number of infinite, connected components of $\gimel\setminus K$, where we run over all finite $K\subset \gimel$. Hence, one can define the number of arms of the DLA as the number of ends of the graph $\gimel=A_\infty$. Due to the fact that $A_\infty$ can be written as the limit of the sets $A_{R^a}\cap B_R$, we can erase a finite set $K$ in finite times and only look on the dynamics after such times.

\begin{conjecture}
	There exists $\theta_0\in (0,2\pi)$ such that for any $\theta\in (0,\theta_0)$, $A_\infty$ has only one arm.
\end{conjecture}

\begin{remark}
Computer simulations seem to suggest that $\theta_0$ is smaller than $\pi/4$. See Figure \ref{fig:arms}.
\end{remark}

\begin{figure}[h!]
\includegraphics[height=4cm]{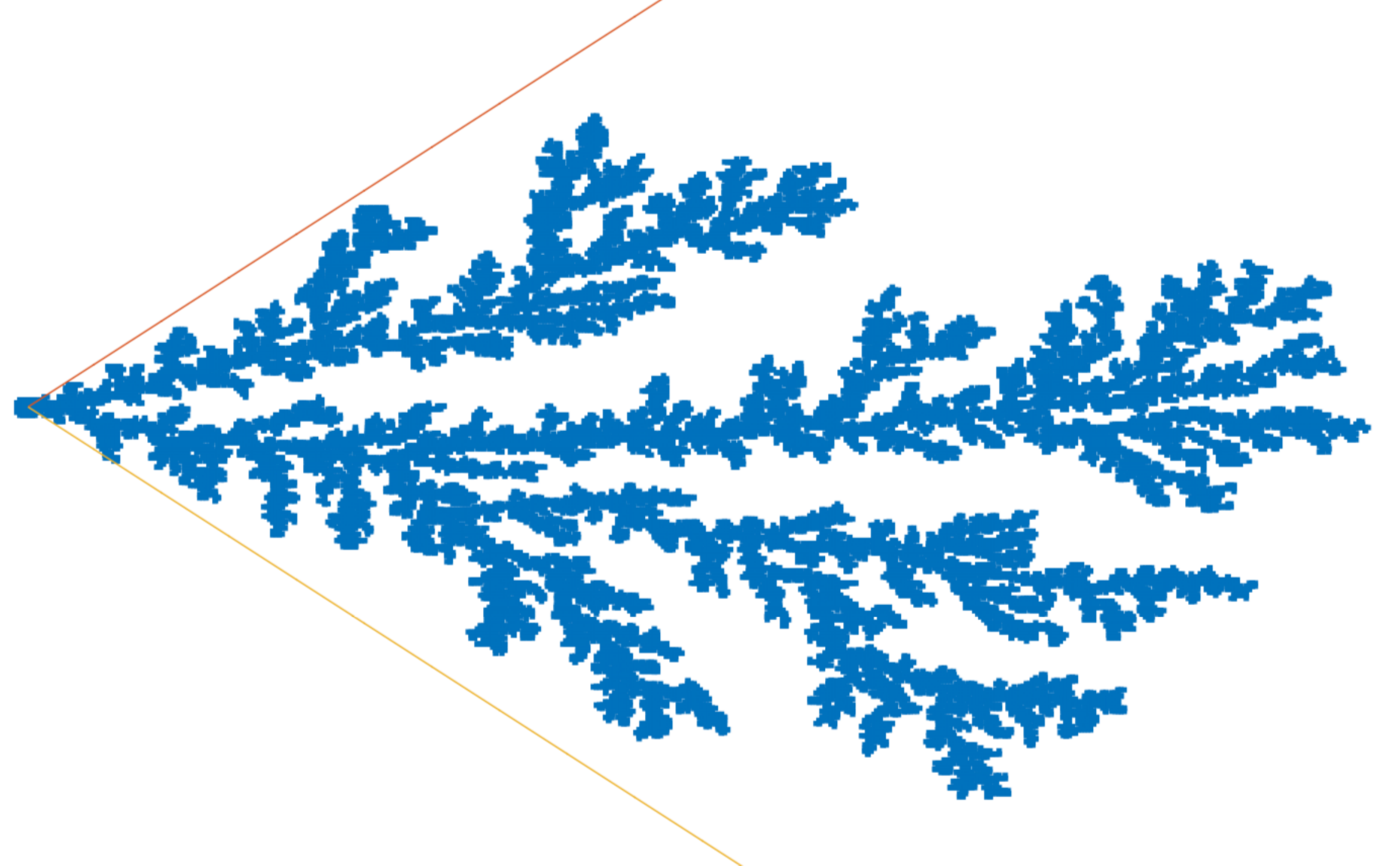}
\includegraphics[height=4cm]{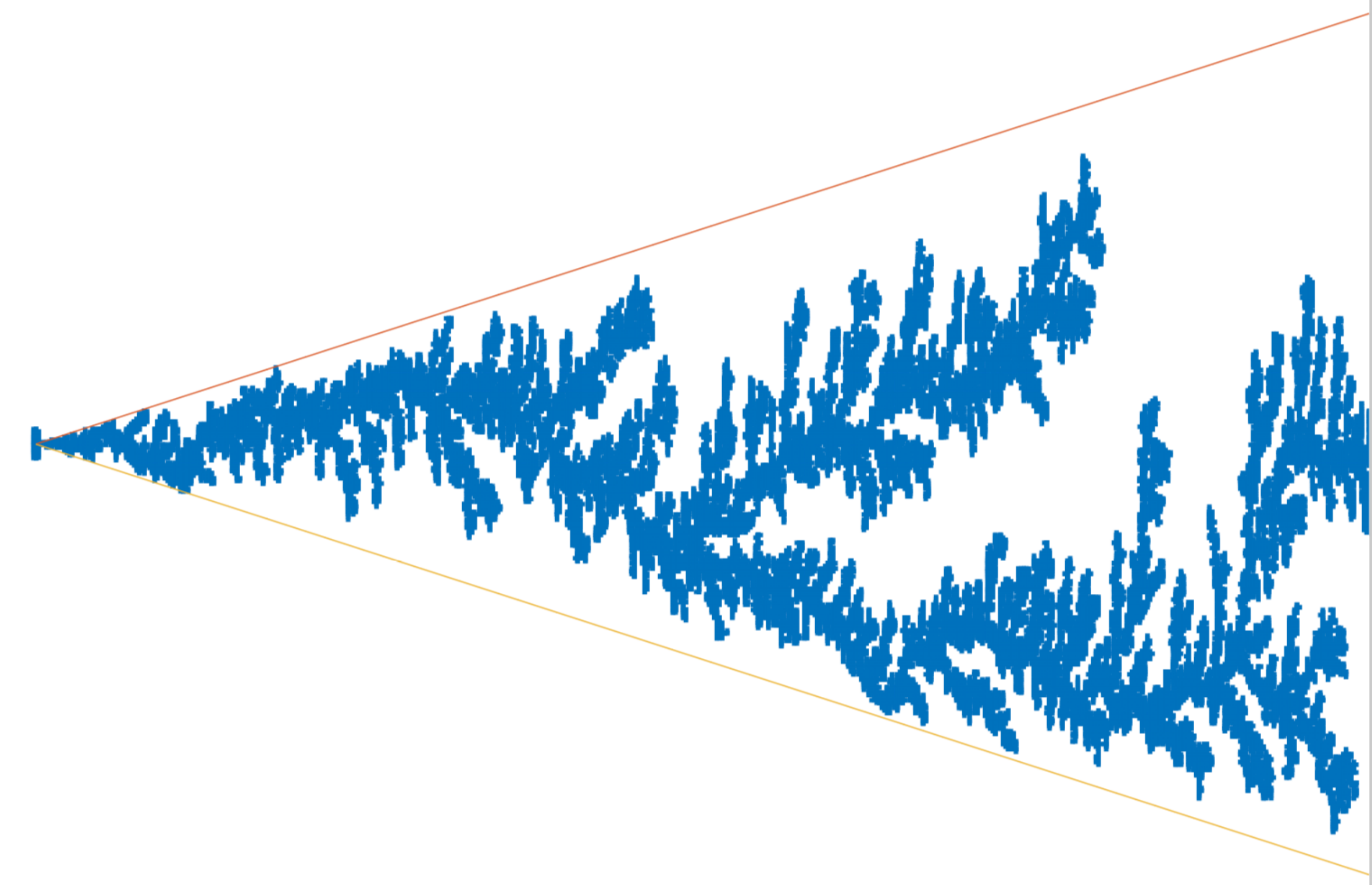}
\caption{DLA in wedges of angles $\pi/4$ and $\pi/40$ (angles look larger because of horizontal stretching for visability). \label{fig:arms}}
\end{figure}

From our results one can deduce that if in a wedge of angle smaller than $\pi/4$, there are two connected components and one is behind the other by a polynomial order, then asymptotically the smaller component will cease to grow. However, without a lower bound on the growth rate this is not enough to prove the conjecture. That being said we would like to suggest one more ambitious conjecture for the DLA in a wedge. Define the growth rate of $(A_n)_{n\geq 0}$, denote $\mathrm{gr}((A_n)_{n\geq 0})$ by 
\[
	\mathrm{gr}((A_n)_{n\geq 0})=\sup\left\{\beta\ge1/2:\limsup_{n\rightarrow\infty}\frac{\text{diam}(A_n)}{n^\beta}>0\right\}
,\]
where diam stands for the diameter of the set in the Euclidean distance. 
\begin{conjecture}
	The growth rate is $\BP_{\theta_1,\theta_2}$-almost surely a constant and as $\theta\rightarrow 0$, it converges to $1$.
\end{conjecture}


\section{Formal definition of the DLA process in a wedge}\label{sec:pre}

This section is devoted to the formal definition of the DLA process in a wedge. In particular, we state a result regarding the existence of the harmonic measure from infinity in it, whose proof we postpone to Appendix \ref{section: harmonic}. 

For $x=(x_1,x_2)\in\BZ^2$ we denote by $\|x\|_2=\sqrt{x_1^2+x_2^2}$ it Euclidean distance from the origin. Fix $-\pi/2\leq \theta_1<\theta_2\leq \pi/2$. Recall that for $x\in W_{\theta_1,\theta_2}$ we denote by $\mathbf{P}_{\theta_1,\theta_2}^x$ the law of a simple random walk $(S_n)_{n\geq 0}$ in $W_{\theta_1,\theta_2}$ staring from $x$, and for $A\subset W_{\theta_1,\theta_2}$ and $y\in\BZ^2$, define $\CH_A(x,y)=\mathbf{P}_{\theta_1,\theta_2}^x(S_{\tau_{A}}=y)$.

\begin{theorem}\label{theorem:harm}
	For every $A\subset W_{\theta_1,\theta_2}$ and $y\in\BZ^2$, the following limit, called the harmonic measure of $A$ from infinity exists
	\[
		\CH_A^\infty(y):=\lim_{\|x\|_2\to \infty} \CH_A(x,y).
	\]
\end{theorem}

Unlike the analouge problem in the whole plain $\BZ^2$, c.f. \cite[Proposition 6.6.1]{lawler2010random}, the  existence of the limiting harmonic measure in a wedge is highly nontrivial and delicate. The main issue here is that, as of right now, there seems to be no discrete Green function approximations on the wedge (or part of it) that is precise enough to match our needs in applying that same approach. A detailed proof for the existence of the limit can be found in Appendix \ref{section: harmonic}. 

Combining Theorem \ref{thm:Beu_estimate} and Theorem \ref{theorem:harm}  we obtain

\begin{corollary}\label{cor:beurlingharm}
Fix $-\pi/2\le \theta_1<\theta_2\leq \pi/2$. For every $\ep>0$, there exists $M\in\BN$ and $C\in (0,\infty)$ such that for all $r,L\in\BN$ satisfying $r\geq M$ and $L/r\geq M$ and every connected subset $A\subset W_{\theta_1,\theta_2}$, such that $W^r_{\theta_1,\theta_2}\subset A$ and $A\cap \partial W_{\theta_1,\theta_2}^L \neq \emptyset$.
	\begin{equation}
	\CH_{\partial A}^\infty(\partial W^r_{\theta_1,\theta_2}) \le C\left(\frac{r}{L} \right)^{\frac{\pi}{2(\theta_2-\theta_1)}-\ep}r\log L.
	\end{equation}
\end{corollary}

Using the existence of the limit $\CH_{\partial A}^\infty(y)$, we can now formally define the DLA in the wedge, denoted $(A_n)_{n\geq 0}$, to be a sequence of random subsets of $W_{\theta_1,\theta_2}$ such that 
\begin{equation}
	A_0 =\{(0,0)\} \qquad \text{and}\qquad 
	A_{n+1} = A_n \uplus \{a_{n+1}\}\,,
\end{equation}
where, given $A_n$, the vertex $a_{n+1}\in W_{\theta_1,\theta_2}$ is sampled according to $\CH_{\partial A_n}^\infty(\cdot)$. Note that $(A_n)_{n\geq 0}$ is an increasing family of subsets in $W_{\theta_1,\theta_2}$, and we denote its limit by $A_\infty = \bigcup_{n=0}^\infty A_n$.

\
\section{Proof of Theorem \ref{thm:main}}

Throughout the remainder of the paper we fix $-\pi/2\leq \theta_1<\theta_2\leq \pi/2$. In this section we assume Corollary \ref{cor:beurlingharm} and turn to prove Theorem \ref{thm:main}. Let $M>0$ and $R\in\BN$ such that $R>M$. Denote by $(L_n)_{n\geq 0}$ a sequence such that $L_0\geq  R$ and $L_{n+1}/L_n\geq M$ for every $n\geq 1$, and by $\si_i$ the sequence of successive first exit times of $(A_n)_{n\geq 0}$ from $B_{L_i}$, namely,
\[
	\si_i=\inf\{n\geq 0: A_n\not\subset B_{L_i}\}.
\]

\begin{lemma}\label{cor:backprob} 
	Let $\varepsilon>0$. For all $M$ sufficiently large (depending only on $\varepsilon$) there exists a constant $C=C(M,\varepsilon)\in (0,\infty)$ such that for all $R>0$ sufficiently large 
\[
	\ev[|(A_\infty\setminus A_{\si_1}) \cap B_R|]
	\le C\sum_{i=1}^\infty L_{i+1}^2\bigg(\frac{R}{L_i}\bigg)^{\frac{\pi}{2(\theta_2-\theta_1)}-\epsilon}R\log L_i.
\]
\end{lemma}

\begin{proof}
	We rewrite $\ev[|(A_\infty\setminus A_{\tau_1}) \cap B_R|]$ as 
\[
	\ev[|(A_\infty\setminus A_{\si_1}) \cap B_R|]
	=\sum_{i=1}^\infty\ev\big[\big|(A_{\si_{i+1}}\setminus A_{\si_{i}})\cap  B_R\big|\big]\,,
\]
and turn to bound each of the terms on the right hand side. 

	Fix $i\geq 1$. Note that $|W_{\theta_1,\theta_2}^{L_{i+1}}|\leq |B_{L_{i+1}}|\leq \pi L_{i+1}^2$ and therefore $\si_{i+1}\leq \pi L_{i+1}^2$. For every $\sigma_i\leq j\leq \sigma_{i+1}$, denote $B_j=A_j\cup W^R_{\theta_1,\theta_2}$ and note that the random set $B_j$ is a connected set containing $W^R_{\theta_1,\theta_2}$ such that $B_j\cap \partial W^{L_i}_{\theta_1,\theta_2}\neq \emptyset$. Therefore, by Corollary \ref{cor:beurlingharm}, for every $\si_i\leq j<\si_{i+1}$,
\[
	\CH_{ \partial A_j}^\infty(W^R_{\theta_1,\theta_2})\leq \CH_{ \partial B_j}^\infty(\partial W^R_{\theta_1,\theta_2}) \le C\left(\frac{R}{L_i} \right)^{\frac{\pi}{2(\theta_2-\theta_1)}-\ep}R\log L_i\,,
\]	
where for the first inequality we used the fact that a particle starting from a sufficiently far point must hit $\partial W^R_{\theta_1,\theta_2}$ before hitting $B_R$. 

Consequently, the random variable $(A_{\tau_{i+1}}\setminus A_{\tau_{i}})\cap  B_R$ is stochastically dominated by a sum of $\pi L_{i+1}^2$ independent Bernoulli random variables with success probability $C(R/L_i)^{\frac{\pi}{2(\theta_2-\theta_1)}-\epsilon}R \log L_i$, whose expectation is $CL_{i+1}^2 (R/L_i)^{\frac{\pi}{2(\theta_2-\theta_1)}-\epsilon}R\log L_i$. 
\end{proof}

Fix some $b>1$ and choose the sequence $L_i:=M^i R^b$. By Lemma \ref{cor:backprob}, for every $\varepsilon>0$ and $M,R>0$ sufficiently large 
\begin{align}
	&\ev[|(A_\infty\setminus A_{\tau_1}) \cap B_R|]\nonumber\\
	&\qquad \le C\sum_{i=1}^\infty L_{i+1}^2\bigg(\frac{R}{L_i}\bigg)^{\frac{\pi}{2(\theta_2-\theta_1)}-\ep}R\log(L_i)\\
	&\qquad =CR^{2b+1-(b-1)\big(\frac{\pi}{2(\theta_2-\theta_1)}-\ep\big)}\sum_{i=1}^\infty M^2(i\log M +b\log R)M^{\big(2+\ep-\frac{\pi}{2(\theta_2-\theta_1)}\big)i}.\nonumber
\end{align}

Since $\ep>0$ can be chosen to be arbitrary small, if $\theta_2-\theta_1<\pi/4$, then the sum on the right hand side is finite and equals 
\[
		\ev[|(A_\infty\setminus A_{\tau_1}) \cap B_R|]\leq CR^{2b+1-(b-1)\big(\frac{\pi}{2(\theta_2-\theta_1)}-\ep\big)}(\log M +b\log R)M^{\big(4+\ep-\frac{\pi}{2(\theta_2-\theta_1)}\big)}\,.
\]
Furthermore, for every fixed $b>1$, the expectation goes to zero as $R\rightarrow\infty$ as soon as 
\begin{equation}\label{eq:difference_condition_1}
	\theta_2-\theta_1<\frac{\pi(b-1)}{4b+2}.
\end{equation}
As $b\to \infty$, this assumption coincides with the previous one, namely $\theta_2-\theta_1<\pi/4$. In fact, for $\theta_1,\theta_2$ such that $\theta_2-\theta_1<\pi/4$, the power of $R$ on the right hand side can be made arbitrarily small by increasing $b$, and in particular, if  
\begin{equation}\label{eq:difference_condition_2}
	\theta_2-\theta_1<\frac{\pi(b-1)}{4(b+1)}
\end{equation}
the power is strictly smaller than $-1$. Assuming \eqref{eq:difference_condition_2} it follows that from the Markov inequality that
\[
	\sum_{R=1}^{\infty} \mathbf{P}_{\theta_1,\theta_2}(|A_\infty\setminus A_{\si_1(MR^b)}\cap B_R|\geq 1)<\infty\,, 
\]
and therefore, using the Borel Cantelli lemma, that $\BP_{\theta_1,\theta_2}$-almost surely, for all $R>0$ sufficiently large, there are no particles hitting $B_R$ after time $\si_1=\si_1(MR^b)$, namely, after the aggregate reaches distance $MR^b$. Taking $a>2b$ and noting that for sufficiently large $R$ we have $\si_1 \leq |B_{L_1}|\leq \pi L_1^2 =M^2R^{2b}<R^a$ the result follows.\hfill\qed

\section{Discrete Beurling estimate in a wedge}

The goal of this section is to provee Theorem \ref{thm:Beu_estimate}. We start by describing the proof strategy.
\begin{enumerate}[(i)]
\item Applying the strong Markov property, and the fact that a random walk starting from radius $R$ must hit radius $L$ before $r$, we conclude that it suffices to consider random walks starting from radius $L$.
\item Using time reversibility of the random walk, we rewrite the hitting probability as the ratio between the escape probability from radius $r$ to radius $L$ while avoiding $A$, and the probability starting from radius $L$ to hit $A$ before returning to the starting point.
\item We bound the probability to hit $A$ before returning to the starting point is bounded from below using the theory of electrical networks. 
\item Finally, we bound the escape probability from radius $r$ to radius $L$ from above with the help of the invariance principle and geometric observations on the set $A$ for which the escape probability is maximal.
\end{enumerate}

\subsection{Proof of Theorem \ref{thm:Beu_estimate} - Reversibility and key lemmas}

Recall that for a set $B\subset W_{\theta,\theta_2}$ we denote by $\tau_B^+=\inf\{n\geq 1~:~S_n\in B\}$ the first return time to $B$ and let $\tau_B=\inf\{n\geq 0~:~S_n\in B\}$ be the first hitting time of $B$. 

We start our proof by rewriting the hitting probabilities from far away appearing in \eqref{eq:Beu_est_1} as escaping probabilities. First, note that for all $y\in W^{r}_{\theta_1,\theta_2}$ and $x\in \partial W^R_{\theta_1,\theta_2}$ the hitting probability $\mathbf{P}_{\theta_1,\theta_2}^x(S_{\tau_{\partial A}}=y)$, can only increase if we replace  $A$ with $A\cap W^{L}_{\theta_1,\theta_2}$. Thus, without loss of generality, we can assume that $A\subset W^{L}_{\theta_1,\theta_2}$. For a set $B\subset W_{\theta_1,\theta_2}$, define $\overline{B}=B\cup\partial B$ to be its clousre and note that the assumption on $A$ implies $\overline{A}\subset \overline{W}^L_{\theta_1,\theta_2}$.

For any $R\gg L$ and $x\in \partial W^{R}_{\theta_1,\theta_2}$, observe that a random walk starting at $x$ must hit $\partial W^{L}_{\theta_1,\theta_2}$ before hitting $\partial A$. Thus, by the strong Markov property, for any $y\in W^{r}_{\theta_1,\theta_2}$
\begin{equation}
\label{strong Markov 1}
\begin{aligned}
	\mathbf{P}_{\theta_1,\theta_2}^x(\tau_{\partial A}=\tau_y)&=\sum_{u\in \partial W^{L}_{\theta_1,\theta_2}} \mathbf{P}_{\theta_1,\theta_2}^x \big(S_{\tau_{\partial W^L_{\theta_1,\theta_2}}}=u\big)\mathbf{P}_{\theta_1,\theta_2}^u(\tau_{\partial A}=\tau_y)\\
&\le \sup_{u\in \partial W^{L}_{\theta_1,\theta_2}\setminus \overline{ A}}\mathbf{P}_{\theta_1,\theta_2}^u(\tau_{\partial A}=\tau_y). 
\end{aligned}
\end{equation}
For any $u\in \partial W^{L}_{\theta_1,\theta_2}\setminus \overline{ A}$, we now rewrite the hitting probability from $u$ to $y$ as the escaping probabiity from $y$  to $u$.  This is done using the reversibility property of simple random walk on graphs. The method of replacing the hitting probability with the escaping probability was used in H. Kesten \cite{MR915132} work on the DLA (see also  \cite{procaccia2017harmonic} for the case of a domain with a boundary). 

First, note that for each $u\in \partial W^{L}_{\theta_1,\theta_2}\setminus \overline{A}$ and $y\in \partial W^{r}_{\theta_1,\theta_2}$ , 
\begin{equation}\label{inverse_time_0}
\begin{aligned}
	\mathbf{P}^u_{\theta_1, \theta_2}(\tau_{\partial A}=\tau_{y})&=\sum_{n=1}^\infty \mathbf{P}^u_{\theta_1, \theta_2}(\tau_{\partial 	A}=\tau_y=n)\\
&=\sum_{n=1}^\infty \mathbf{P}^u_{\theta_1, \theta_2}(S_1\notin \overline{A},\cdots, S_{n-1}\notin \overline{A},  S_n=y). 
\end{aligned}
\end{equation}
Here we used the fact that $A$ is a connected set and that $u\notin \overline{A}$. As a result, the whole path must stay outside of the set $\overline{A}$ until it first hit $y$. 

Let $\Gamma^{A,n}_{u,y}$ be the collection of all paths $\gamma = (x_0,x_1,x_2,\cdots, x_n)$ of length $n$ in $W_{\theta_1, \theta_2}$ such that $x_0=u$, $x_n=y$ and $x_1,\cdots, x_{n-1}\notin \overline{A}$. For $\gamma=(x_0,\ldots,x_n)\in \Gamma^{A,n}(u,y)$  denote by $\widehat{\gamma} = (x_n,\ldots,x_0)\in \Gamma^{n,A}(y,u)$ the path in the reverse direction. Then, the reversibility of simple random walk implies
\[
	\mathbf{P}_{\theta_1,\theta_2}^u((S_0,\ldots, S_n)=\gamma)
	= \frac{\deg(y)}{\deg(u)}	\mathbf{P}_{\theta_1,\theta_2}^y((S_0,\ldots, S_n)=\widehat{\gamma})\qquad \forall \gamma\in \Gamma^{A,n}_{u,y}\,,
\]
where for $z\in W_{\theta_1,\theta_2}$ we denote by $\deg(z)$ its degree in the graph $W_{\theta_1,\theta_2}$. Since $\deg(z)\in \{1,2,3,4\}$ for every $z\in W_{\theta_1,\theta_2}$ we conclude that 
\begin{equation}\label{inverse_time_2}
	\mathbf{P}_{\theta_1,\theta_2}^x((S_0,\ldots, S_n)=\gamma)
	\leq 4 \mathbf{P}_{\theta_1,\theta_2}^y((S_0,\ldots, S_n)=\widehat{\gamma})\qquad \forall \gamma\in \Gamma^{A,n}_{u,y}\,.
\end{equation}

Combining \eqref{inverse_time_0} and \eqref{inverse_time_2}, 
\begin{equation}\label{inverse_time_3}
\begin{aligned}
	\mathbf{P}^u_{\theta_1, \theta_2}(\tau_{\partial A}=\tau_{y})&=\sum_{n=1}^\infty \tbP^u_{\theta_1, \theta_2}(S_1\notin \overline{A},\cdots, S_{n-1}\notin \overline{A},  S_n=y )\\
&\le 4\sum_{n=1}^\infty \tbP^y_{\theta_1, \theta_2}(S_1\notin \overline{A},\cdots, S_{n-1}\notin \overline{A},  S_n=u)\\
&=4 \tbE^y_{\theta_1, \theta_2}\big[|\{n\in [0,\tau_{\overline{A}}^+) ~:~ S_n=u\}|\big]. 
\end{aligned}
\end{equation}

By the strong Markov property for the stopping time $\tau_{\overline{A}}^+\wedge \tau^+_{u}$
\begin{equation}\label{inverse_time_3.5}
\begin{aligned}
&	\tbE^y_{\theta_1, \theta_2}\big[|\{n\in [0,\tau^+_{\overline{A}}) ~:~ S_n=u\}|\big] \\
	 =&\mathbf{P}^y_{\theta_1,\theta_2}(\tau^+_{u}< \tau^+_{\overline{A}})\cdot \mathbf{E}^y_{\theta_1, \theta_2}\big[|\{n\in [0,\tau^+_{\overline{A}}) ~:~ S_n=u\}|\big]
=\frac{\mathbf{P}^y_{\theta_1,\theta_2}(\tau^+_{u}<\tau^+_{\overline{A}})}{\mathbf{P}^u_{\theta_1, \theta_2}(\tau^+_{\overline{A}}\leq \tau^+_u)}\,,\\ 
\end{aligned}
\end{equation}
where in the last step we used the fact that the number of visits to $u$ before hitting $A$ when starting from $x$ is a geometric random variables with parameter $\mathbf{P}^u_{\theta_1, \theta_2}(\tau^+_{\overline{A}}<\tau^+_u)$. 

Combining \eqref{strong Markov 1}, \eqref{inverse_time_3} and \eqref{inverse_time_3.5} together with the fact that $|\partial W^r_{\theta_1,\theta_2}|\leq Cr$ for some universal constant $C\in (0,\infty)$ gives
\[
\begin{aligned}
\mathbf{P}^x_{\theta_1, \theta_2}\big(\tau_{ \partial W^r_{\theta_1, \theta_2}}=\tau_{\partial A}\big)&=\sum_{y\in \partial W^r_{\theta_1,\theta_2}}\mathbf{P}^x_{\theta_1, \theta_2}(\tau_{y}=\tau_{\partial A})\\
&\le \sum_{y\in \partial W^r_{\theta_1,\theta_2}} \sup_{u\in \partial W^{L}_{\theta_1,\theta_2}\setminus \overline{A}}\mathbf{P}_{\theta_1,\theta_2}^u(\tau_{\overline{A}}=\tau_y)\\
&\le Cr \sup_{y\in \partial W^r_{\theta_1,\theta_2}} \sup_{u\in \partial W^{L}_{\theta_1,\theta_2}\setminus \overline{A}}\mathbf{P}_{\theta_1,\theta_2}^u( \tau_{\overline{A}}=\tau_y)\\
&= Cr \sup_{y\in \partial W^r_{\theta_1,\theta_2}} \sup_{u\in \partial W^{L}_{\theta_1,\theta_2}\setminus \overline{A}}\frac{\mathbf{P}^y_{\theta_1,\theta_2}(\tau^+_{u}<\tau^+_{\overline{A}})}{\mathbf{P}^u_{\theta_1, \theta_2}(\tau^+_{\overline{A}}\leq \tau^+_u)}. 
\end{aligned}
\]

Consequently, in order to prove Theorem \ref{thm:Beu_estimate}, it suffices to show that there is a constant $C<\infty$ such that for every $\ep>0$, $y\in  \partial W^r_{\theta_1,\theta_2}$ and every $u\in \partial W^{L}_{\theta_1,\theta_2}\setminus \overline{A}$,
\begin{equation}\label{what we need}
\frac{\mathbf{P}^y_{\theta_1,\theta_2}(\tau^+_{u}<\tau^+_{\overline{A}})}{\mathbf{P}^u_{\theta_1, \theta_2}(\tau^+_{\overline{A}}\leq \tau^+_u)}\le C\left(\frac{r}{L} \right)^{\frac{\pi}{2(\theta_2-\theta_1)}-\ep} \log L\,,
\end{equation}
for sufficiently large $L$. The last inequality is an immediate corollary of the following two lemmas. \qed

\begin{lemma}\label{lemma_upper_bound_return}
	There exists a constant $c\in (0,\infty)$ independent of $r$ and $L$ such that 
$$
	\mathbf{P}^u_{\theta_1, \theta_2}(\tau^+_{\overline{A}}\leq \tau^+_u) \geq \frac{c}{\log L}	
$$
uniformly for all $A\subset W^L_{\theta_1,\theta_2}$ as in Theorem \ref{thm:Beu_estimate} and all $u\in \partial W^L_{\theta_1,\theta_2}\setminus \overline{A}$. 
\end{lemma}

\begin{lemma}\label{lemma_escape_2}
	For every $\varepsilon>0$ there exists a constant $C\in (0,\infty)$ such that for every $r,L$ sufficiently large, every $y\in \partial W_{\theta_1,\theta_2} ^r$, every $A\subset W^L_{\theta_1,\theta_2}$ as in Theorem \ref{thm:Beu_estimate} and every $u\in \partial W_{\theta_1,\theta_2} ^L\setminus \overline{A}$
$$
	\mathbf{P}^y_{\theta_1,\theta_2}(\tau^+_{u}<\tau^+_{\overline{A}})\le \mathbf{P}^y_{\theta_1, \theta_2}(\tau^+_{\partial W^L_{\theta_1,\theta_2}}<\tau^+_{\overline{A}})\leq C\left(\frac{r}{L}\right)^{\frac{\pi}{\theta_2-\theta_1}-\varepsilon}\,.
$$
\end{lemma}

The proof of Lemma \ref{lemma_upper_bound_return} is presented in Subsection \ref{subsec:1} and the proof of  Lemma \ref{lemma_escape_2} can be found in Subsection \ref{subsec:2}. 

\subsection{Proof of Lemma \ref{lemma_upper_bound_return}}\label{subsec:1}

We will use the theory of electrical networks in order to estimate the effective resistance from $u$ to $0$ in the graph $W_{\theta_1,\theta_2}$. Since $0\in W_{\theta_1,\theta_2}^r\subset  A$ it follows that $\tau^+_{\overline{A}}\leq  \tau_0$, and therefore there was a mistake here with the definition of effective resistance, but the proof is still correct.
\[
	\mathbf{P}^u_{\theta_1, \theta_2}(\tau^+_{\overline{A}}<\tau^+_u)
	\geq \mathbf{P}^u_{\theta_1, \theta_2}(\tau_{0}<\tau^+_u)
\]

Recall that the effective resistance from a vertex $v\subset W_{\theta_1,\theta_2}$ to $Z\subset W_{\theta_1,\theta_2}$ is given by 
\[
	\CR(v\to Z) = \frac{1}{\deg(v) \BP_{\theta_1,\theta_2}^v(\tau^+_{v}  < \tau_Z) }.
\]	

Repeating the argument in \cite{LP17}[Proposition 2.15] for the graph $W_{\theta_1,\theta_2}$ instead of $\BZ^2$ shows that for every $-\pi/2\leq \theta_1<\theta_2\leq \pi/2$, there exists a constant $C\in (0,\infty)$ so that 
\[
	\CR(0\to x) \in [C^{-1} \log \|x\|_2 , C\log \|x\|_2],\qquad \forall x\in W_{\theta_1,\theta_2}\,. 
\]
Noting that $\deg(x)\in \{1,2,3,4\}$ for every $x\in W_{\theta_1,\theta_2}$, the result follows. \qed

\subsection{Proof of Lemma \ref{lemma_escape_2} - Escaping probability to distance $L$}\label{subsec:2}

Having completed the proof of Lemma \ref{lemma_upper_bound_return}, we turn to the proof of Lemma \ref{lemma_escape_2}. The proof of the latter contains several subclaims: (i) Finding the worst possile choice for the set $A$, (ii) explicit calculation for the continuous counterpart of the upper bound obtained in step (i) and (iii) using the invariance principle to compare the discrete and the continuous probabilities. We now turn to implement this strategy.

\subsubsection{Proof of Lemma \ref{lemma_escape_2} part (i) - The worst choice for $A$}

We start by showing that the worst choice for $A$, namely the set which maximizes the probability $\mathbf{P}^y_{\theta_1, \theta_2}(\tau^+_{\partial W^L_{\theta_1,\theta_2}}<\tau^+_{\overline{A}})$ among all sets $A\subset W^L_{\theta_1,\theta_2}$ as in Theorem \ref{thm:Beu_estimate} is given by one of the lines along the wedge boundary.

To this end we define the {\bf discrete upper and lower boundaries} of $W_{\theta_1, \theta_2}$ by 
\[
	\Gamma^u_{\theta_1, \theta_2}=\Big\{x\in W_{\theta_1, \theta_2} ~:~ \exists y\in\BZ^2 \text{ satisfying } \|y-x\|_\infty=1 \text{ and }\arctan(y_2/y_1)>\theta_2  \Big\}
\]
and 
\[
	\Gamma^l_{\theta_1, \theta_2}=\Big\{x\in W_{\theta_1, \theta_2} ~:~ \exists y\in\BZ^2 \text{ satisfying } \|y-x\|_\infty=1 \text{ and }\arctan(y_2/y_1)<\theta_1  \Big\},
\]
where we use the same convention regarding the function $\arctan$ as in the introduction. Also, for $0<r<L$ and $\alpha \in \{u,l\}$, we denote 
\[
	\Gamma^{\alpha,r,L}=\Gamma^{\alpha,r,L}_{\theta_1,\theta_2} :=\partial W^r_{\theta_1,\theta_2}\cup (\Gamma^{\alpha}_{\theta_1,\theta_2}\cap (W^{L}_{\theta_1,\theta_2} \setminus W^r_{\theta_1,\theta_2})) .
\]

\begin{lemma}\label{lemma push to boundary}
For all sufficiently large $r$, all sufficiently large $L> r$, every set $A$ as in Theorem \ref{thm:Beu_estimate} and every $y\in \partial W^r_{\theta_1,\theta_2}$

\begin{equation}\label{push to boundary}
\begin{aligned}
	&\mathbf{P}^y_{\theta_1, \theta_2}(\tau^+_{\partial W^L_{\theta_1,\theta_2}}<\tau^+_{\overline{A}})\\
&\qquad \le \max\big\{\mathbf{P}^y_{\theta_1, \theta_2}\big(\tau^+_{\partial W^L_{\theta_1,\theta_2}}<\tau^+_{\Gamma^{u,r,L}}\big), 
\mathbf{P}^y_{\theta_1, \theta_2}\big(\tau^+_{\partial W^L_{\theta_1,\theta_2}}< \tau^+_{\Gamma^{l,r,L}}\big) \big\}. 
\end{aligned}
\end{equation}
\end{lemma}
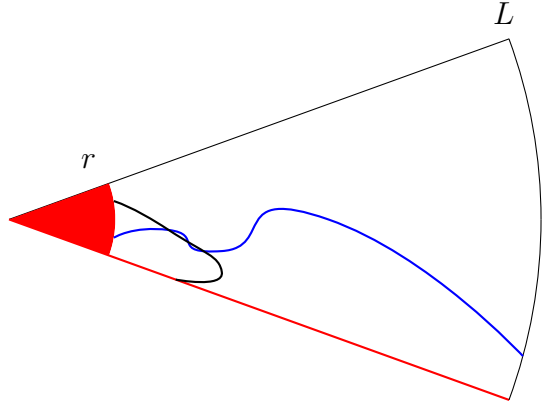
\begin{figure}
\input{maxhitset}
\caption{Illustration of proof argument for Lemma \ref{lemma push to boundary}}
\end{figure}
\begin{proof}
	Recall that $ W^{r}_{\theta_1, \theta_2}\subset A$,  $\partial W^{L}_{\theta_1, \theta_2}\cap A\not=\emptyset$ and that $A$ is a connected subset within the wedge $W_{\theta_1,\theta_2}$. Therefore there exists a connected path $\gamma=(x^0,x^1,\ldots,x^k)$ within $\overline{A}$ from $ W^{r}_{\theta_1, \theta_2}$ to $\partial W^{L}_{\theta_1, \theta_2}$. Without loss of generality we can assume that $x^{k}\in \partial W^{L}_{\theta_1, \theta_2}$, $x^0\in \partial W^r_{\theta_1,\theta_2}$ and that $x^{1},x^{2},\cdots, x^{k-1}\in W_{\theta_1, \theta_2}\cap ( \overline{W}^r_{\theta_1,\theta_2}\cup \partial W^{L}_{\theta_1, \theta_2})^c$.
	
Let $A'= \partial W^{r}_{\theta_1, \theta_2}\cup \gamma$. Since $A'\subset \overline{A}$ we know that $\tau^+_{\overline{A}}\leq \tau^+_{A'}$ and therefore
\[
 \mathbf{P}^y_{\theta_1, \theta_2}(\tau^+_{\partial W^L_{\theta_1,\theta_2}}<\tau^+_{\overline{A}})
 \leq \mathbf{P}^y_{\theta_1, \theta_2}(\tau^+_{\partial W^L_{\theta_1,\theta_2}}<\tau^+_{A'}).
\]

Next, we turn to compare the hitting probability in $A'$ to the hitting probability in $\Gamma^{u,r,L}$ and $\Gamma^{l,r,L}$. We separate the proof into three cases according to the position of $y$ with respect to the path $\gamma$:
\begin{itemize}
	\item[(1)] The point $y$ satisfies $\arctan(y_2/y_1)>\arctan(x^0_2/x^0_1)$.
	\item[(2)] The point $y$ satisfies $\arctan(y_2/y_1)<\arctan(x^0_2/x^0_1)$.
	\item[(3)] The point $y$ satisfies $y=x^0$. 
\end{itemize}

We start with case $(1)$, since $W_{\theta_1,\theta_2}$ is a connected, planar graph and $\gamma$ is a connected path in it from $\partial W^r_{\theta_1,\theta_2}$ to $\partial W^L_{\theta_1,\theta_2}$, every path from $y$ to $\Gamma^{l,r,L}$ must hit $A'$. In particular, paths starting in $y$ that hit $\Gamma^{l,r,L}$ before hitting $\partial W^L_{\theta_1,\theta_2}$ must hit $A'$ before hitting $\partial W^L_{\theta_1,\theta_2}$. This implies that $\mathbf{P}^y_{\theta_1, \theta_2}(\tau^+_{\Gamma^{l,r,L}}\leq \tau^+_{\partial W^L_{\theta_1,\theta_2}})\leq \mathbf{P}^y_{\theta_1, \theta_2}(\tau^+_{A'}\leq \tau^+_{\partial W^L_{\theta_1,\theta_2}})$ and therefore 
the required inequality. Similarly, in case $(2)$, every path from $y$ to $\Gamma^{u,r,L}$ must hit $A'$ and therefore paths hitting $\Gamma^{u,r,L}$ before $\partial W^L_{\theta_1,\theta_2}$ must also hit $A'$ before $\partial W^L_{\theta_1,\theta_2}$. Hence $\mathbf{P}^y_{\theta_1, \theta_2}(\tau^+_{\partial W^L_{\theta_1,\theta_2}}\le \tau^+_{A'})\le \mathbf{P}^y_{\theta_1, \theta_2}(\tau^+_{\partial W^L_{\theta_1,\theta_2}}\le \tau^+_{\Gamma^{u,r,L}})$.

Finally, we turn to deal with case $(3)$. After one step of the random walk we have $\mathbf{P}^y_{\theta_1, \theta_2}(\tau^+_{\partial W^L_{\theta_1,\theta_2}}\le \tau^+_{A'})= \frac{1}{\deg(y)}\sum_{z\in W_{\theta_1,\theta_2} ~s.t. \|z-y\|_1=1}\mathbf{P}^z_{\theta_1, \theta_2}(\tau_{\partial W^L_{\theta_1,\theta_2}}\le \tau_{A'})$.
Since for $z\in A'$ we have $\mathbf{P}^z_{\theta_1, \theta_2}(\tau_{\partial W^L_{\theta_1,\theta_2}}\le \tau_{A'})=0$ and for $z\notin A'$ we can repeat the argument in $(1)$ and $(2)$ we conclude that
\[
\begin{aligned}
	&\mathbf{P}^y_{\theta_1, \theta_2}(\tau^+_{\partial W^L_{\theta_1,\theta_2}} \le \tau^+_{A'}) \\
	&\qquad \leq \frac{1}{\deg(y)}\sum_{\substack{z\in W_{\theta_1,\theta_2} \\ \|z-y\|_1=1}}\max\big\{\mathbf{P}^z_{\theta_1, \theta_2}\big(\tau^+_{\partial W^L_{\theta_1,\theta_2}}<\tau^+_{\Gamma^{u,r,L}}\big), \mathbf{P}^z_{\theta_1, \theta_2}\big(\tau^+_{\partial W^L_{\theta_1,\theta_2}}< \tau^+_{\Gamma^{l,r,L}}\big) \big\}\\
	&\qquad =\max\big\{\mathbf{P}^y_{\theta_1, \theta_2}\big(\tau^+_{\partial W^L_{\theta_1,\theta_2}}<\tau^+_{\Gamma^{u,r,L}}\big), 
\mathbf{P}^y_{\theta_1, \theta_2}\big(\tau^+_{\partial W^L_{\theta_1,\theta_2}}< \tau^+_{\Gamma^{l,r,L}}\big) \big\}\,,
\end{aligned}
\]
as required.
\end{proof}

Due to the last lemma it is enough to obtain bounds on 
\[
	\mathbf{P}^y_{\theta_1, \theta_2}(\tau^+_{\partial W^L_{\theta_1,\theta_2}}\le \tau^+_{\Gamma^{l,0,L}})\qquad \text{ and }\qquad 	\mathbf{P}^y_{\theta_1, \theta_2}(\tau^+_{\partial W^L_{\theta_1,\theta_2}}\le \tau^+_{\Gamma^{u,0,L}}) 
\]
for all $y\in \partial W^r_{\theta_1,\theta_2}$ and all sufficiently large $r$ and $L$. We focus here on the estimation for $\Gamma^{l,0,L}$, the bound for $\Gamma^{u,0,L}$ is obtained in the same manner.

Let $K$ be a constant to be chosen later on, for $i\geq 0$, define $M_i = rK^i$ and let $N=N(r,L,K)$ be the largest integer such that $M_N\leq L$. Then
\[
	\mathbf{P}^y_{\theta_1, \theta_2}\Big(\tau^+_{\partial W^L_{\theta_1,\theta_2}}\le \tau^+_{\Gamma^{l,0,L}}\Big)\le \mathbf{P}^y_{\theta_1, \theta_2}\Big(\tau^+_{\partial W^{M_N}_{\theta_1, \theta_2}}\le \tau^+_{\Gamma^{l,0,L}}\Big)\,,
\]
and by strong Markov property 
\begin{equation}\label{product}
\mathbf{P}^y_{\theta_1, \theta_2}\Big(\tau^+_{\partial W^{M_N}_{\theta_1, \theta_2}}\le \tau^+_{\Gamma^{l,0,L}}\Big)\le \prod_{i=0}^{N-1}
\sup_{z\in \partial W^{M_i}_{\theta_1, \theta_2}} \mathbf{P}^z_{\theta_1, \theta_2}\Big(\tau_{\partial W^{M_{i+1}}_{\theta_1, \theta_2}}\le \tau_{\Gamma^{l,M_{i},M_{i+1}}}\Big).
\end{equation}

In order to estimate each of the probabilities on the right hand side we first turn to evaluate their continuous counterpart.

\subsubsection{Proof of Lemma \ref{lemma_escape_2} part (ii) - Calculating the continuous analogue of the probability}

Define the continuous wedge between the angles $-\pi/2\leq \theta_1<\theta_2\leq \pi/2$ by 
\[
	\CW_{\theta_1, \theta_2}=\{(x_1,x_2)\in \BR^2 ~:~ \arctan(x_2/x_1)\in [\theta_1, \theta_2]\}\,,
\]
with the same convention as in the discrete case regarding $\arctan$. Furthermore, define the intersection of $\CW_{\theta_1,\theta_2}$ with the ball of radius $K$ around the origin by
\[
	\CW^K_{\theta_1, \theta_2}= \CW_{\theta_1,\theta_2}\cap \{x\in \BR^2 ~:~ \|x\|_2\leq K\}
\]
and the lower, upper and front boundaries of $\CW_{\theta_1,\theta_2}^K$ respectively, by 
\[
	\partial^l\CW_{\theta_1,\theta_2}^K = \{x\in \CW_{\theta_1,\theta_2}^K ~:~ \arctan(x_2/x_1)=\theta_1\},
\]
\[
	\partial^u\CW_{\theta_1,\theta_2}^K = \{x\in \CW_{\theta_1,\theta_2}^K ~:~ \arctan(x_2/x_1)=\theta_2\},
\]
and
\[
	\partial^f\CW_{\theta_1,\theta_2}^K = \{x\in \CW_{\theta_1,\theta_2}^K ~:~ \|x\|_2=K\}\,.
\]

Let $(B(t))_{t\geq 0}=(B^K_{\theta_1,\theta_2}(t))_{t\geq}$ denote a reflected Brownian motion in $\CW^K_{\theta_1,\theta_2}$ and denote by $T_K$ the hitting time of $\partial^l\CW_{\theta_1,\theta_2}^K\cup \partial^f\CW_{\theta_1,\theta_2}^K$. With a slight abuse of notation, we use $\mathbf{P}^x_{\theta_1,\theta_2}$ to denote the law of $(B(t))_{t\geq 0}$ with starting point $x$ as well. In this subsection, our goal is to estimate the continuous analogue of the discrete probability 
\[
	\mathbf{P}^z_{\theta_1, \theta_2}\big(\tau_{\partial W^{K L}_{\theta_1, \theta_2}}\le \tau_{\Gamma^{l,L,KL}}\big),\qquad \forall z\in \partial W^L_{\theta_1,\theta_2}
\]
given by 
\[
	\mathbf{P}^x_{\theta_1,\theta_2}(|B(T_K)|=K)\,,
\]
for all $x$ of the form $(\cos\theta,\sin\theta)$ with $\theta\in[\theta_1,\theta_2]$. 

Due to the rotation invariance of (reflected) Brownian motion we can replace the angles $\theta_1,\theta_2$ by the angles $-(\theta_2-\theta_1),0$. Furthermore, using the reflection principle and denoting by $S_K$ the hitting time of $\partial \CW^K_{\theta_1,\theta_2} = \partial \CW^l_{\theta_1,\theta_2}\cup \partial \CW^u_{\theta_1,\theta_2}\cup \partial \CW^f_{\theta_1,\theta_2}$ we conclude that 
\[
	\mathbf{P}^{(\cos\theta,\sin\theta)}_{\theta_1,\theta_2}(|B(T_K)|=K) = \mathbf{P}^{(\cos(-\theta+\theta_1),\sin(-\theta+\theta_1))}_{-\varphi,\varphi}(|B(S_K)|=K)\,,
\]
where we denoted $\varphi = \theta_2-\theta_1$. See Figure \ref{fig:reflectfar}.

\begin{figure}[h!]
	\input{escapefar}\qquad\qquad 
	\input{escapefar_reflect}
	\caption{Reflecting the Dirichlet boundary conditions \label{fig:reflectfar}}
\end{figure}
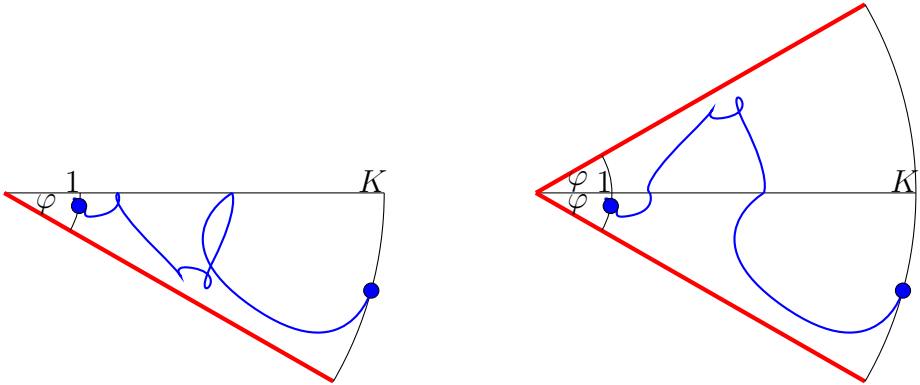

The probability on the right hand side was calculated in \cite{morters2010brownian} for the case $\theta=\theta_1$

\begin{lemma}[\cite{morters2010brownian} Theorem 7.24]\label{lem:cont_prob} Let $\varphi\in (0,\pi]$ and $K>1$. Then 

 \[
 	\mathbf{P}^{(1,0)}_{-\varphi,\varphi}(|B(S_K)|=K) = \frac{2}{\pi}\arctan\bigg(\frac{2K^{\pi/2\varphi}}{K^{\pi/\varphi}-1}\bigg).
 \]
\end{lemma}

In addition, using conformal maps (or alternatively the Beurling estimation) one can verify that 
\begin{equation}\label{Maximum boundary}
	\mathbf{P}^x_{-\varphi,\varphi}(|B(S_K)|=K)\leq \mathbf{P}^{(1,0)}_{-\varphi,\varphi}(|B(S_K)|=K),
\end{equation}
for all $x$ of the form $(\cos\psi,sin\psi)$ with $\psi\in [-\varphi,\varphi]$.

\subsubsection{Proof of Lemma \ref{lemma_escape_2} part (iii) - from continuous to discrete and completion of the proof}

In this subsection, we use Lemma \ref{lem:cont_prob} in order to find an upper bound on the discrete probability 
\[
\mathbf{P}^z_{\theta_1, \theta_2}\Big(\tau_{\partial W^{K L}_{\theta_1, \theta_2}}\le \tau_{\Gamma^{l,L,KL}}\Big),\qquad \forall z\in \partial W^L_{\theta_1,\theta_2}
\]

\begin{lemma}\label{lemma connect}
	Let $K>0$. For every $\varepsilon>0$, there exists a constant $L_0\in (0,\infty)$ such that for all $L\ge L_0$, and all $z\in \partial  W^L_{\theta_1,\theta_2}$
\begin{equation}\label{continuous probability 6}
	\mathbf{P}^z_{\theta_1, \theta_2}\Big(\tau_{\partial W^{KL}_{\theta_1, \theta_2}}\le \tau_{\Gamma^{l,L,KL}_{\theta_1, \theta_2}}\Big)\le \mathbf{P}^{(1,0)}_{-\varphi,\varphi}(|B(S_{K})|=K)+\varepsilon,
\end{equation}
where as before, we denote $\varphi=\theta_2-\theta_1$.
\end{lemma}

The main ingredient in the proof of lemma \ref{lemma connect} is the invariance principle for simple random walk in a wedge. 

\begin{lemma}[Lemma 2.1 in \cite{2006}, Theorem 6.3 of \cite{1971}] 
\label{lemma weak converge}
Let $D$ be a bounded connected open set with an analytic (smooth) boundary in $\RR^d$, $d\ge 2$ and let $D_\ep=\ep \ZZ^d\cap D$. Suppose that $x_\ep \in D_\ep$ and $x_\ep\to x_0\in \overline{D}$ as  $\ep\to0$. Let $\{W^\ep_t, \ t\ge 0\}$ be simple random walk on $D_\ep$ with $W^\ep_0=x_\ep$. Then $(W^\ep_t)_{t\ge 0}$ converge weakly to  reflected Brownian motion on $\overline{D}$ starting from $x_0$ as $\ep\to0$.  
\end{lemma}

\begin{remark}
In \cite{2006}, it was assumed throughout the paper that $D$ is a bounded, connected, open set with analytic boundary. The analyticity assumption is made for technical reason, needed in proving their main result. However, it is noted in the paper that the lemma above is derived from Theorem 6.3 of \cite{1971}, which holds for smooth regions as well. 
\end{remark}

\begin{proof}[Proof of Lemma \ref{lemma connect}]
	Suppose the lemma does not hold. Then, there exists $\ep_0>0$ and an increasing sequence $(L_n)_{n\geq 1}$ going to infinity together with a sequence of points $(z_n)_{n\geq 1}$ such that $z_n\in \partial  W^{L_n}_{\theta_1,\theta_2}$ for all $n\geq 1$ such that
\begin{equation}\label{contradiction 1}
\mathbf{P}^{z_n}_{\theta_1, \theta_2}\left(\tau_{\partial W^{KL_n}_{\theta_1, \theta_2}}\le \tau_{\Gamma^{l,L_n,KL_n}_{\theta_1, \theta_2}}\right)\ge \mathbf{P}^{(1,0)}_{-\varphi,\varphi}(|B(S_{K})|=K)+\varepsilon_0.
\end{equation}
Noting that $|z_n/L_n|\to 1$ as $n\to\infty$, it follows that there exists a subsequence $k_n$ such that 
\[
	\lim_{n\to\infty} z_{k_n}/L_{k_n}= z_0\in \{x\in\CW_{\theta_1,\theta_2} ~:~ \|x\|_2=1\}. 
\]
Consequently, by Lemma \ref{lemma weak converge} and \eqref{Maximum boundary} we have 
\begin{equation}\label{continuous probability 7}
\begin{aligned}
\lim_{n\to\infty} \mathbf{P}^{z_{k_n}}_{\theta_1, \theta_2}\Big(\tau_{\partial W^{KL_{k_n}}_{\theta_1, \theta_2}}\le \tau_{\Gamma^{l,L_{k_n},KL_{k_n}}_{\theta_1, \theta_2}}\Big)&=\mathbf{P}^{z_{0}}_{\theta_1, \theta_2}\big(\big|B_k^{\theta_1,\theta_2}(\infty)\big|=K\big)\\
&\le  \mathbf{P}^{(1,0)}_{-\varphi,\varphi}(|B(S_{K})|=K)
\end{aligned}
\end{equation}
contradicting \eqref{contradiction 1}. 
\end{proof}

At this point we have all the ingredients needed in order to complete the proof of Theorem \ref{thm:Beu_estimate}. Recalling Lemma \ref{Maximum boundary}, we observe that there exists a universal constant $C\in (0,\infty)$ such that
\begin{equation}\label{continuous probability 4}
\mathbf{P}^{(1,0)}_{-\varphi,\varphi}(|B(S_K)|=K) = \frac{2}{\pi}\arctan\bigg(\frac{2K^{\pi/2\varphi}}{K^{\pi/\varphi}-1}\bigg)\le  \frac{C}{K^{\pi/2\varphi}},\qquad \forall K\geq 2.
\end{equation}
Furthermore, since $\lim_{K\to\infty}\log_K(C+1)\to 0$, there exists $K_0\in (1,\infty)$ such that 
\begin{equation}\label{continuous probability 10}
	0<\log_{K_0}(C+1)<\frac{\ep}{4}.
\end{equation}

Taking $\ep=\ep_{K_0}=K_0^{-\pi/2\varphi}$ in Lemma \ref{lemma connect}, we conclude that exists $R_0\in (0,\infty)$ such that for all $r\ge R_0$ and all $z\in  \partial  W^r_{\theta_1,\theta_2}$
\begin{equation}\label{continuous probability 11}
\begin{aligned}
	\mathbf{P}^z_{\theta_1, \theta_2}\Big(\tau_{\partial W^{rK_0}_{\theta_1, \theta_2}}\le \tau_{B^{l,0,rK_0}_{\theta_1, \theta_2}}\Big)&\le \mathbf{P}^{(1,0)}_{-\varphi,\varphi}(|B(S_K)|=K)+\ep_{K_0}\leq \frac{C+1}{K_0^{\pi/2\varphi}}. 
\end{aligned}
\end{equation}

Next, recall that $N_{K_0}$ was defined to be the largest integer $n$ such that $rK_0^n\le L$, which implies that $N_{K_0}=\lfloor\log_{K_0}(L/r)\rfloor$. Consequently, for all $r\ge R_0$, and $L$ sufficiently large so that 
\[
\bigg(\frac{L}{r}\bigg)^{\ep/2}\ge K_0^{\pi/2\varphi}
\]
we have 
\begin{equation}\label{product 2}
\begin{aligned}
	&\mathbf{P}^y_{\theta_1, \theta_2}\Big(\tau_{\partial W^{M_{N(K_0)}}_{\theta_1, \theta_2}}\le \tau_{B^{l,0,L}_{\theta_1, \theta_2}}\Big)\le \prod_{i=0}^{N-1}
\sup_{z\in \partial W^{M_i}_{\theta_1, \theta_2}} \mathbf{P}^z_{\theta_1, \theta_2}\left(\tau_{\partial W^{M_{i+1}}_{\theta_1, \theta_2}}\le \tau_{\Gamma^{l,M_{i},M_{i+1}}}\right)\\
&\qquad \le \bigg(\frac{C+1}{K_0^{\pi/2\varphi}}\bigg)^{\log_{K_0}(L/r)-1}
\leq  K_0^{\pi/2\delta} \bigg(\frac{r}{L} \bigg)^{\pi/2\varphi-\log_{K_0}(C+1)}\\
&\qquad \le  K_0^{\pi/2\delta} \bigg(\frac{r}{L} \bigg)^{\pi/2\varphi-\frac{1}{4}\ep}
 \le \bigg(\frac{r}{L} \bigg)^{\pi/2\varphi-\ep}.
\end{aligned}
\end{equation}
Thus, the proof of Lemma \ref{lemma_escape_2} is complete. \qed

\appendix
\section{Existence of infinite harmonic measure}
\label{section: harmonic} 

In this section, we prove Theorem \ref{theorem:harm}. The convergence is proved by showing that for any $y\in W_{\theta_1,\theta_2}$ and any sequence $(x_n)_{n\geq 1}$ in $W_{\theta_1,\theta_2}$ such that $\lim_{n\to\infty}\|x_n\|_2=\infty$, the sequence $(\CH_A(x_n,y))_{n\geq 1}$ is Cauchy. 

Let $(x_n)_{n\geq 1}$ be a sequence as above. Since $A$ is finite one can find $r>0$ such that $A\subset W^r_{\theta_1,\theta_2}$ and thus $\CH_A(x,y)=0$ for all $y\notin W^R_{\theta_1,\theta_2}$ and $x\in W_{\theta_1,\theta_2}$. Hence, we can restrict attention to $y\in W^r_{\theta_1,\theta_2}$. Since $\lim_{n\to\infty}\|x_n\|_2=\infty$, we can assume without loss of generality that $\|x_n\|_2>r$ for all $n\geq 1$ 

Note that for any $m,n\in\BN$ such that $\|x_n\|_2<\|x_m\|_2$, a random walk starting in $x_m$ must hit $\partial W^{\|x_n\|_2}_{\theta_1,\theta_2}$ before hitting $A$. Thus it is enough to prove that for every $y\in W^{r}_{\theta_1,\theta_2}$  
\begin{equation}\label{eq:harmdifference}
		\lim_{R\to\infty}\max_{x_1,x_2\in \partial W^R_{\theta_1,\theta_2}}|\CH_A(x_1,y)-\CH_A(x_2,y)|=0.
\end{equation} 
As mentioned in Section \ref{sec:pre}, there is no discrete Green function approximation on the wedge which is accurate enough to allow us to follow the proof outline of H. Kesten in $\BZ^2$. Instead, we will prove the result using the following strategy 
\begin{enumerate}
	\item Show that the number of steps needed for a random walk, starting from $\partial W^R_{\theta_1,\theta_2}$, to reach $A$ is asymptotically bigger than $R^2$. 
	\item Show that the mixing time for a random walk started from $x\in \partial W^R_{\theta_1,\theta_2}$ is much smaller than the hitting time of $A$ and therefore, using a coupling argument, that two random walks starting from $x$ and $x'$ in $\partial W^R_{\theta_1,\theta_2}$ respectively, will coincide with high probability before hitting $A$. 
\end{enumerate}
Note that carrying out the strategy above requires careful choices of parameters in the proof. This is the content of the following subsections. 

\subsection{Coupling, two key propositions and the proof of Theorem \ref{theorem:harm}}

For $R>0$ and $x_1,x_2\in \partial W^R_{\theta_1,\theta_2}$, denote by $\tbP^{x_1,x_2}_R$ a coupling of two continuous time Markov processes $(B^R_1(t),B^R_2(t))_{t\geq 0}$ each with state space $(W_{\theta_1,\theta_2}/R)$ defined as follows:
\begin{enumerate}[(a)]
	\item $B^R_1(t)$ is a continuous time, simple random walk on $W_{\theta_1,\theta_2}/R$, starting at $x_1/R$, with fixed jump rate $2R^2$.
\item $B^R_2(t)$ is a continuous time, simple random walk on $W_{\theta_1,\theta_2}/R$, starting at $x_2/R$, with fixed jump rate $2R^2$.
\item $(B^R_1(t))_{t\geq 0}$ and $(B^R_2(t))_{t\geq 0}$ are coupled according to the maximum coupling, see \cite[Appendix A.4.2]{lawler2010random}.
\end{enumerate}

For $i\in \{1,2\}$, we define $(S^R_i(n))_{n\geq 1}$ to be the embedded, discrete time, simple random walk in $(B^R_i(t))_{t\geq 0}$, and for $s\geq 0$, denote by $N^R_i(s)$, the number of jumps made by the Markov process $(B_i^R(t))_{t\geq 0}$ up to time $s$. It follows from the definitions above that $B^R_i(s)=S^R_i(N^R_i(s))$ for $i\in \{1,2\}$, $R>0$ and $s\geq 0$. 

Denoting by $\tau^i_A=\inf\{t\geq 0 ~:~ S_i^R(t)\in A/R\}$ the hitting time of $(S_i^R(t))_{t\geq 0}$ in $A/R$, it follows from the definition of $\CH_A(\cdot,\cdot)$ that for every $y\in W^r_{\theta_1,\theta_2}$
\[
	\CH_A(X_i,y)
	=\mathbf{P}^{x_i}\big(S_i(\tau_A)=y\big)	
	=\mathbf{P}^{x_1,x_2}_R\big(B^R_i(\tau^{i}_A)=y\big),\qquad \forall i\in \{1,2\}.
\]

Define the stopping time 
\[
	\CT = \inf\{t\geq 0 ~:~ B_1(t)=B_2(t)\}\,,
\]
and note that from the definition of the coupling and the stopping time $B_1(s)=B_2(s)$ for all $s\geq \CT$. 

For $T_0>0$ and $R>0$, define the event 

\[
I_{T_0,R}=\{\tau^{1}_A>T_0, \, \tau^{2}_A>T_0, \,\CT\leq T_0\}. 
\]
Then, by the Markov property 
\[
	\{B^R_1(\tau^{1}_A)=y\}\cap I_{T_0,R}=\{B^R_2(\tau^{2}_A)=y\}\cap I_{T_0,R}\,,
\]
and therefore
\begin{equation}\label{change to mixing}
\begin{aligned}
&|\CH_A(x_1,y)-\CH_A(x_2,y)|\\
&\qquad =|\mathbf{P}^{x_1,x_2}_R\big(\{B^R_1(\tau^{1}_A)=y\}\cap I^c_{T_0,R}\big)-\mathbf{P}^{x_1,x_2}_R\big(\{B^R_2(\tau^{2}_A)=y\}\cap I^c_{T_0,R}\big)\\
&\qquad \le \mathbf{P}^{x_1,x_2}_R( I^c_{T_0,R})
\leq \mathbf{P}^{x_1,x_2}_R(\tau_A^1\leq T_0)+ \mathbf{P}^{x_1,x_2}_R(\tau_A^2\leq T_0) + \mathbf{P}^{x_1,x_2}_R(\CT>T_0)\\
\end{aligned}
\end{equation}

Also, note that for $i\in \{1,2\}$
\begin{equation}\label{eq:LDP}
\begin{aligned}
	\mathbf{P}^{x_1,x_2}_R(\tau_A^i\leq T_0) &\leq \mathbf{P}^{x_1,x_2}_R(N_i(T_0)\ge 4T_0R^2)+\mathbf{P}^{x_1,x_2}_R(N_i(\tau^{i}_A)\le 4T_0R^2)\\
	&\leq e^{-2(\log(4)-1)T_0R^2} + \tbP^{x_i}_{\theta_1,\theta_2}(\tau_A\leq 4T_0R^2)\,,
\end{aligned}
\end{equation}
where in the last inequality we used large deviation estimate for the random variable $N_1(T_0)\sim \mathrm{Pois}(2T_0R^2)$.

Combining all of the above, we conclude that for every $T_0>0$, $R>0$, $x_1,x_2\in W^R_{\theta_1,\theta_2}$ and $y\in W^r_{\theta_1,\theta_2}$
\[
\begin{aligned}
	|\CH_A(x_1,y)-\CH_A(x_2,y)|&\leq 2\max_{x\in W^R_{\theta_1,\theta_2}}\tbP^{x}_{\theta_1,\theta_2}(\tau_A\leq 4T_0R^2)+\mathbf{P}^{x_1,x_2}_R(\CT>T_0)\\&+2e^{-2(\log(4)-1)T_0R^2}.
\end{aligned}
\]

Consequently, the proof of Theorem \ref{theorem:harm} is an immediate consequence of the following two propositions:

\begin{Proposition}\label{proposition hit A}
	For every finite set $A\in W_{\theta_1,\theta_2}$ and every $T\in (0,\infty)$ 
\[
	\lim_{R\to\infty}\max_{x\in \partial W^R_{\theta_1,\theta_2}}\mathbf{P}^x_{\theta_1, \theta_2}(\tau_A\le TR^2)=0\,.
\]
\end{Proposition}

\begin{Proposition}\label{proposition_2}
	For every finite set $A\in W_{\theta_1,\theta_2}$
\[
	\lim_{T_0\to\infty}\lim_{R\to\infty}\max_{x_1,x_2\in \partial W^R_{\theta_1,\theta_2}}\mathbf{P}^{x_1,x_2}_R(\CT>T_0)=0\,.
\]
\end{Proposition}

\begin{proof}[Proof of Theorem \ref{theorem:harm}]
	Let $A\subset W_{\theta_1,\theta_2}$ be a finite set. The discussion above combined with Proposition \ref{proposition hit A} and Proposition \ref{proposition_2} implies that \eqref{eq:harmdifference} holds. As a result, the limit $\lim_{\|x\|_2\to\infty}\CH_A(x,y)$ exists for every $y\in W_{\theta_1,\theta_2}$ and thus the existence of the Harmonic measure follows. 
\end{proof}

\subsection{Lower bound on the hitting time - Proof of Proposition \ref{proposition hit A}}
 
We start with some results for the continuous analogue of reflected Brownian motion in the continuous wedge $\CW_{\theta_1, \theta_2}$, whose law when starting in $u\in\CW_{\theta_1,\theta_2}$ we denote by $\widehat{\tbP}^u_{\theta_1,\theta_2}$. For $L>0$ denote by $\si_L$ the hitting time of the reflected Brownian motion in $\partial \CW_{\theta_1,\theta_2}^L$.

\begin{lemma}\label{lemma_compare_out}
	For every $\ep>0$, $C>1$ and $u\in \CW_{\theta_1, \theta_2}$ such that $\|u\|_2=1$
\begin{equation}\label{compare_out}
	\widehat{\mathbf{P}}^u_{\theta_1, \theta_2}(\si_{C^{-1}}<\si_{C^\ep})=\frac{\ep}{1+\ep}\,.
\end{equation}
\end{lemma}

\begin{proof}
	The proof follows from the fact that $\log|x|$ is the Green function in $\BR^2$ and therefore, if $(B(t))_{t\geq 0}$ is a standard two-dimensional Brownian motion, then $(\log|B(t)|)_{t\geq 0}$ is a martingale. 
	
			Indeed, note that the reflected Brownian motion in a smooth region is conformally invariant up to a time change, c.f. Theorem 9.3 of \cite{fields_medal_paper}. By the conformal mapping theorem we can map the reflected we can map the wedge into $\BC\setminus [0,\infty)$, which transforms the Brownian motion in $\CW_{\theta_1, \theta_2}$ to a Brownian motion in $\BC^2$ reflected on the line $\{(x,0) ~:~ x\in\BR\}$. Note that the original event $\{\si_{C^{-1}}<\sigma_{C^\varepsilon}\}$ is mapped under this transformation to the event $\{\si_{C^{-2\pi/(\theta_2-\theta_1)}}<\si_{C^{2\pi\ep/(\theta_2-\theta_1)}}\}$. Next, observe that by the reflection principle, the reflection on the line $\{(x,0) ~:~ x\in\BR\}$ does not change the probability of the event $\{\si_{C^{-2\pi/(\theta_2-\theta_1)}}<\si_{C^{2\pi\ep/(\theta_2-\theta_1)}}\}$. 
	
	Due to the fact that $(\log(|B(t)|))_{t\geq 0}$ is a martingale, where $(B(t))_{t\geq 0}$ is a standard two-dimensional Brownian motion. It follows from the optional stopping theorem for the stopping time $\si:=\min\{\si_{C^{-2\pi/(\theta_2-\theta_1)}},\,\si_{C^{2\pi\ep/(\theta_2-\theta_1)}}\}$ that 
\[
	-\frac{2\pi}{\theta_2-\theta_1}\widehat{\mathbf{P}}^u_{\theta_1, \theta_2}(\si_{C^{-1}}<\si_{C^\ep})\log(C)+\frac{2\pi\ep}{\theta_2-\theta_1}(1-\widehat{\mathbf{P}}^u_{\theta_1, \theta_2}(\si_{C^{-1}}<\si_{C^\ep}))\log C=0
\]
which proves the result.
\end{proof}

\begin{lemma}\label{lemma_time_out_0}
	For every $\ep>0$, there is a constant $C_\ep\in (0,\infty)$ such that for all $u\in \CW_{\theta_1,\theta_2}$ satisfying $\|u\|_2=1$ and all $T>0$ 
\[
	\widehat{\mathbf{P}}^u_{\theta_1, \theta_2}(\si_{T^{1/2+\ep}}\le T)\le \frac{C_\ep}{T}. 
\]
\end{lemma} 
\begin{proof}
	It suffices to prove the result for all sufficiently large $T$. Recall that the time change of the process under the conformal map is given, due to Ito's formula, by
\[
	\zeta(t)=\int_0^t |f'(\widehat{B}(s))|^2ds\,,
\]
where $f$ is the conformal map from the the upper half plane to $\CW_{\theta_1, \theta_2}$, given by 
\[
	f(x)=x^\varphi,
\]
with $\varphi:=(\theta_2-\theta_1)/\pi$, and $\widehat{B}(s)$ is the reflected Brownian motion in the upper half plane. For $T>0$, define
\[
	\widehat{\sigma}_T=\inf\{t\geq 0 ~:~ \|\widehat{B}(t)\|_2=T\}. 
\]
Then
\begin{equation}\label{time_out_0_1}
\begin{aligned}
	&\widehat{\mathbf{P}}^u_{\theta_1, \theta_2}(\si_{T^{1/2+\ep}}\le T)\\
	&\qquad \le \mathbf{P}^{u^{1/\varphi}}(\widehat{\sigma}_{T^{(1/2+\ep)/\varphi}}\le T^{(1+\ep/2)/\varphi})+ \mathbf{P}^{u^{1/\varphi}}(\zeta( T^{(1+\ep/2)/\varphi}) \le T). 
\end{aligned}
\end{equation}

Starting from the first term on the right hand side of \eqref{time_out_0_1}, note that $\|u^{1/\varphi}\|_2=1$ for every $u=(u_1,u_2)$ such that $\|u\|_2=1$. Also, observe that a reflected Brownian motion in the upper half plane can be constructed by replacing the $y-$coordinate of a standard 2-dimensional Brownian motion by its absolute value. Since, $\|\widehat{B}(t)\|_2 \geq T$ implies that one of the coordinates of $\widehat{B}(t)$ is bigger than $T/2$ it follows that 
\begin{equation}\label{time_out_0_2}
\begin{aligned}
	\mathbf{P}^{u^{1/\varphi}}&(\widehat{\sigma}_{T^{(1/2+\ep)/\varphi}}\le T^{(1+\ep/2)/\varphi})\\
&\le 2\sup_{|a|\leq 1}\mathbf{P}^{a}\Big(\max_{t\le T^{(1+\ep/2)/\varphi}}|B^1(t)|\ge T^{(1/2+\ep)/\varphi}/2\Big)\,,
\end{aligned}
\end{equation}
where $B^1(t)$ is a one-dimensional Brownian motion. By reflection principle for one dimensional Brownian motion 
\[
\begin{aligned}
	&\mathbf{P}^{a}\Big(\max_{t\le T^{(1+\ep/2)/\varphi}}|B^1(t)|\ge T^{(1/2+\ep)/\varphi}/2 \Big)\\
	&\qquad\le 4 \mathbf{P}^{a}(B^1(T^{(1+\ep/2)/\varphi})\ge T^{(1/2+\ep)/\varphi}/2)\le 4\exp(-T^{\ep/2})<\frac{1}{2T}
\end{aligned}
\]
for all $|a|\leq 1$ and all $T>0$ sufficiently large. Thus it remains to control the second term on the right hand side of \eqref{time_out_0_1}, and show that for every $\ep>0$ there exists $C_\ep\in (0,\infty)$ such that 
\begin{equation}\label{time_out_0_4}
	\mathbf{P}^{u^{1/\varphi}}\left(\zeta( T^{(1+\ep/2)/\varphi}) \le T\right)\le \frac{C_\ep}{T}. 
\end{equation}
Recalling that $f(x)=x^\varphi$, we have $|f'(\widehat{B}(s))|^2=\varphi^2 |\widehat{B}(s)|^{2(\varphi-1)}$.
Define $\delta_\ep=\varphi\ep/(3+\ep)$. For any $n\in\BN$, using similar argument as in \eqref{time_out_0_2}, we have 
\begin{equation}\label{time_out_0_5}
\begin{aligned}
	& \mathbf{P}^{u^{1/\varphi}}\Big(\max_{t\le n}\|B(t)\|_2\ge n^{(1+\delta_\ep)/2} \Big)\\
&\qquad \le 2\max_{|a|\leq 1}\mathbf{P}^{a}\Big(\max_{t\le n}|B^1(t)|\ge  n^{(1+\delta_\ep)/2}/2 \Big)
\le 8 \exp(-n^{\delta_\ep/4})
\end{aligned}
\end{equation}
for all sufficiently large $n$. Consequently
\begin{equation}\label{time_out_0_6}
	\sum_{n=[ T^{(1+\ep/2)/\varphi}/2]}^\infty \mathbf{P}^{u^{1/\varphi}}\Big(\max_{t\le n}\|B(t)\|_2\ge n^{(1+\delta_\ep)/2}\Big)\le \frac{C_\ep}{T}.
\end{equation}
Thus it suffices to show that $\zeta(T^{1+\ep/2})>T$ on the event
\[
	A=\bigcap_{n=[ T^{(1+\ep/2)/\varphi}/2]}^\infty \Big\{ \max_{t\le n}\|B(t)\|_2< n^{(1+\delta_\ep)/2} \Big\}\,.
\]
This indeed holds since for all $n\ge [ T^{(1+\ep/2)/\varphi}/2]-1$
\[
	\int_{n-1}^{n} \varphi^2 |\hat B(s)|^{2(\varphi-1)} ds>  \varphi^2 n^{(1+\delta_\ep)(\varphi-1)} \ge  \varphi^2 n^{-1+\varphi-\delta_\ep}= \varphi^2 n^{-1+\varphi/(1+\ep/3)}
\]
and therefore
\begin{equation}\label{time_out_0_5}
\begin{aligned}
	\zeta(T^{(1+\ep/2)/\varphi})&\ge \varphi^2 \int_{\frac{T^{(1+\ep/2)/\varphi}}{2}}^{T^{(1+\ep/2)/\varphi}}|\hat B(s)|^{2(\varphi-1)} ds
&\ge \varphi^2 \sum_{n=[ \frac{T^{(1+\ep/2)/\varphi}}{2}]}^{[T^{(1+\ep/2)/\varphi}]-1} n^{-1+\varphi/(1+\ep/3)}\\
&\ge c T^{(1+\ep/2)/(1+\ep/3)}>T.
\end{aligned}
\end{equation}
\end{proof}
With  Lemma \ref{lemma_time_out_0} at hand, by choosing $C=T^{1+\frac{1}{2\ep}}$ one obtains: 
\begin{lemma}\label{lemma_time_out}
	For every $\ep>0$, every $T\in (0,\infty)$ and any $u\in \CW_{\theta_1, \theta_2}$ satisfying $\|u\|_2=1$,
\begin{equation}
	\lim_{C\to \infty}\widehat{\mathbf{P}}^u_{\theta_1, \theta_2}(\si_{C^\ep}\le T)=0\,. 
\end{equation}
\end{lemma}

Next, using the invariance principle and an argument similar to the one in the proof of Lemma \ref{lemma connect}, we prove analogue results to the ones in Lemmas \ref{lemma_compare_out} -  \ref{lemma_time_out}, for simple random walk in $W_{\theta_1,\theta_2}$. 

For $\ep>0$, $R\geq 1$ and $C>0$, let 
\[
(S_n^{R,C,\varepsilon})_{n\geq 0}=\bigg(S_{n\wedge \tau_{\partial W^{2C^\ep R}_{\theta_1,\theta_2}}\wedge \tau_{\partial W^{R/2C}_{\theta_1,\theta_2}}}\bigg)_{n\ge 0}
\]
be a simple random walks in $W_{\theta_1,\theta_2}$, starting from $x_R\in \partial W^R_{\theta_1,\theta_2}$, stopped at the first hitting time of $\partial W^{R/2C}_{\theta_1,\theta_2}$ or $\partial W^{2C^\ep R}_{\theta_1,\theta_2}$. Due to the invariance principle, for every sequence of points $(x_R)_{R\geq 1}$ for which the limit $x_\infty:=\lim_{R\to\infty} x_R/R$ exists, the linear interpolation of $(S_n^{R,C,\varepsilon})_{n\geq 0}$ converges weakly to reflected Brownian motion in the wedge, starting from $x_\infty\in \CW_{\theta_1,\theta_2}$ (satisfying $\|x_\infty\|_2=1$) until the stopping time $\si$. 

\begin{lemma}\label{lemma_continuous_to_discrete}
	For every $\ep>0$, $T\in (0,\infty)$ and $C>1$, there exists $C_\varepsilon\in (0,\infty)$ such that the following holds  
\begin{equation}\label{continuous_to_discrete_1}
	\lim_{R\to\infty}\max_{y\in \partial W^R_{\theta_1,\theta_2}}\mathbf{P}^y_{\theta_1, \theta_2}\big(\tau_{\partial W^{R/C}_{\theta_1,\theta_2}}<\tau_{\partial W^{R C^\ep}_{\theta_1,\theta_2}}\big)\leq \ep,
\end{equation}

\begin{equation}\label{continuous_to_discrete_2}
	\lim_{R\to\infty}\max_{y\in \partial W^R_{\theta_1,\theta_2}}\mathbf{P}^y_{\theta_1, \theta_2}(\tau_{\partial W^{T^{1/2+\ep}R}}\le TR^2)\le \frac{C_\ep}{T},
\end{equation}
and 
\begin{equation}\label{continuous_to_discrete_3}
	\lim_{C\to\infty}\lim_{R\to\infty}\max_{y\in \partial W^R_{\theta_1,\theta_2}}\mathbf{P}^y_{\theta_1, \theta_2}\big(\tau_{\partial W^{C^\ep R}_{\theta_1,\theta_2}}\le TR^2 \big)=0.
\end{equation}
\end{lemma}

\begin{proof}
	As alluded above, we use a similar argument to the one in the proof of Lemma \ref{lemma connect}. Suppose \eqref{continuous_to_discrete_1} does not hold. Then there is a sequence $(R_n)_{n\geq 1}$ going to infinity and a sequence $(x_n)_{n\geq }$ such that for every $n\geq 1$,  $x_n\in  \partial W^{R_n}_{\theta_1,\theta_2}$ and 
\[
	\mathbf{P}^{x_n}_{\theta_1, \theta_2}\big(\tau_{\partial W^{R/C}_{\theta_1,\theta_2}}<\tau_{\partial W^{R C^\ep}_{\theta_1,\theta_2}}\big)> \ep
\]
Since $x_n/R_n$ is a bounded sequence in $\BR^2$, there exists a subsequence $(n_k)_{k\geq 1}$ such that $\lim_{k\to\infty}x_{n_k}/R_{n_k}=x_\infty'\in \CW_{\theta_1, \theta_2}$, satisfying $\|x_\infty'\|_2=1$. Thus by the invariance principle and Lemma Lemma \ref{lemma_compare_out}, 
\[
	\varepsilon\leq \lim_{k\to\infty} \mathbf{P}^{x_{n_k}}_{\theta_1, \theta_2}\big(\tau_{\partial W^{R_{n_k}/C}_{\theta_1,\theta_2}}<\tau_{\partial W^{R_{n_k} C^\ep}_{\theta_1,\theta_2}}\big)
	=\widehat{\mathbf{P}}^{x_\infty'}_{\theta_1, \theta_2}(\si_{C^-1}<\si_{C^\ep})=\frac{\ep}{1+\varepsilon}
\]
which contradicts the assumption. Repeating the argument with Lemma \ref{lemma_time_out_0} and Lemma \ref{lemma_time_out} replacing Lemma \ref{lemma_compare_out}, yields \eqref{continuous_to_discrete_2} and \eqref{continuous_to_discrete_3} respectively. 
\end{proof}

\begin{proof}[Proof of Proposition \ref{proposition hit A}] Since $A$ is a fixed finite set, for every $C>1$ and $R$ sufficiently large (depending only on $C$ and $A$) we have $A\subset W^{R/C}_{\theta_1,\theta_2}$. Hence, for every $x\in \partial W^R_{\theta_1,\theta_2}$ and $R$ sufficiently large 
\[
	\tau_{C^\ep R}\wedge \tau_{R/C}\leq \tau_{R/C}\leq \tau_A,\qquad \tbP^x_{\theta_1,\theta_2}-a.s
\] 
and therrefore, for every $T>0$, $C>1$, $\ep>0$ and $x\in\partial W^R_{\theta_1,\theta_2}$
\[
\begin{aligned}
\mathbf{P}^x_{\theta_1, \theta_2}(\tau_A\le TR^2)&
\le \mathbf{P}^x_{\theta_1, \theta_2}\big( \tau_{\partial W^{C^\ep R}_{\theta_1,\theta_2}}\wedge \tau_{\partial W^{R/C}_{\theta_1,\theta_2}}\le   TR^2\big)\\
&\le \mathbf{P}^x_{\theta_1, \theta_2}\big(\tau_{\partial W^{R/C}_{\theta_1,\theta_2}}<\tau_{\partial W^{C^\ep R}_{\theta_1,\theta_2}} \big)+\mathbf{P}^x_{\theta_1, \theta_2}\big(\tau_{\partial W^{C^\ep R}_{\theta_1,\theta_2}}\le TR^2 \big).
\end{aligned}
\]

Taking the maximum over $x\in \partial W^R_{\theta_1,\theta_2}$, then the limit $R\to \infty$, then the limit $C\to\infty$ and finally the limit $\ep\to 0$, the result follows from Lemma \ref{lemma_continuous_to_discrete}.
\end{proof}

\subsection{Proof of Proposition \ref{proposition_2}}

Since $\tbP^{x_1,x_2}_R$ couples the two random walks via a maximum coupling for Markov chains, it follows that 
\[
\begin{aligned}
	\mathbf{P}^{x_1,x_2}_R(\CT>T_0)&\leq d_{\mathrm{TV}}(B_1^R(T_0),B_2^R(T_0))\\
	&:=\sum_{y\in W_{\theta_1,\theta_2}/R}\big|\mathbf{P}^{x_1,x_2}_R(B^R_1(T_0)=y)-\mathbf{P}^{x_1,x_2}_R(B^R_2(T_0)=y) \big|
\end{aligned}
\]

For $M_0\in (0,\infty)$, one can split the sum over $y$ into two parts, those satisfying $\|y\|_2\geq M_0$ and the ones satisfying $\|y\|_2<M$. We turn to estimate each of the terms separately.

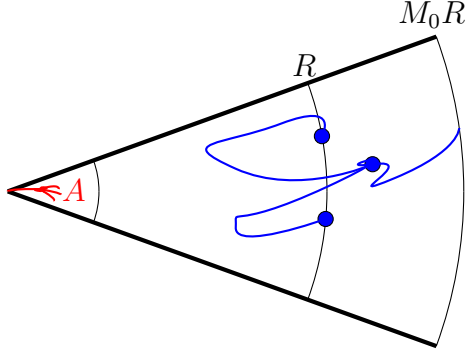
\begin{figure}[h!]
	\input{maxcoupling}
	\caption{Max coupling occurs before hitting $\partial W_{\theta_1,\theta_2}^{M_0 R}$.\label{fig:maxcoupling}}
\end{figure}

For the first sum, note that a similar argument the one in \eqref{eq:LDP} yields 
\[
\begin{aligned}
	&\sum_{\substack{y\in W_{\theta_1,\theta_2}/R \\ \|y\|_2\geq M_0}}\big|\mathbf{P}^{x_1,x_2}_R(B^R_1(T_0)=y)-\mathbf{P}^{x_1,x_2}_R(B^R_2(T_0)=y) \big|\\
 	&\qquad \leq\mathbf{P}^{x_1,x_2}_R(\|B^R_1(T_0)\|_2\geq M_0)+\mathbf{P}^{x_1,x_2}_R(\|B^R_2(T_0)\|_2\geq M_0)\\
 	&\qquad \leq 2e^{-2(\log(4)-1)T_0R^2} + 2\max_{x\in \partial W^R_{\theta_1,\theta_2}}\tbP^{x}_{\theta_1,\theta_2}(\tau_{M_0R}\leq 4T_0R^2).
\end{aligned}
\]

Fixing some $\varepsilon>0$ and defining $M_0= T_0^{1/2+\varepsilon}$, it follows from Lemma \ref{lemma_continuous_to_discrete} (see \eqref{continuous_to_discrete_2}) that the sum is bounded by 
\[
	2e^{-2(\log(4)-1)T_0R^2} + \frac{3C_\varepsilon}{T_0}\leq \frac{4C_\ep}{T_0}\,,
\]
provided $R$ is sufficiently large. 

Next, we turn to estimate the second term, namely 
\begin{equation}\label{eq:whatever}
	\sum_{\substack{y\in \overline{W}_{\theta_1,\theta_2}/R \\ \|y\|_2<M_0} }\big|\mathbf{P}^{x_1,x_2}_R(B^R_1(T_0)=y)-\mathbf{P}^{x_1,x_2}_R(B^R_2(T_0)=y) \big|
\end{equation}

The strategy for bounding the last sum is to use known bounds on the mixing time and total variation distance for random walks on finite graphs, obtained by intersecting scaled version of $\BZ^2$ with some bounded and sufficiently regular domains in $\BR^2$. Note however, that the continuous time, simple random walk in \eqref{eq:whatever} is defined on the cone $W_{\theta_1,\theta_2}/R$, which is not bounded. Thus our first step is to show that the last sum can be well approximated by a corresponding sum for a continuous time, simple random walk in $W_{\theta_1,\theta_2}^{M_0 R}/R$, with $M_0$ chosen to be $T_0^{1/2+\ep}$. 

To this end, for $M_0=T_0^{1/2+\ep}>1$, $R>0$ and $x_1,x_2\in \partial W^R_{\theta_1,\theta_2}$, denote by $\widetilde{\tbP}^{x_1,x_2}_R$ a coupling of two continuous time, simple random walks on $W^{M_0R}_{\theta_1,\theta_2}/R$, denoted $(\widetilde{B}^R_1(t))_{t\geq 0}$ and $(\widetilde{B}^R_2(t))_{t\geq 0}$, defined as follows:
\begin{enumerate}[(a)]
	\item $\widetilde{B}^R_1(t)$ is a continuous time, simple random walk on $W^{M_0R}_{\theta_1,\theta_2}/R$, starting at $x_1/R$, with fixed jump rate of $2R^2$.
\item $\widetilde{B}^R_2(t)$ is a continuous time, simple random walk on $W^{M_0R}_{\theta_1,\theta_2}/R$, starting at $x_2/R$, with fixed jump rate of  $2R^2$.
\item $(\widetilde{B}^R_1(t))_{t\geq 0}$ and $(\widetilde{B}^R_2(t))_{t\geq 0}$ are coupled according to the maximum coupling, see \cite[Appendix A.4.2]{lawler2010random}. 
\end{enumerate}

Furthermore, for $i\in \{1,2\}$ we denote by $\widetilde{N}^R_i(s)$, the number of jumps made by the Markov process $(\widetilde{B}_i^R(t))_{t\geq 0}$ up to time $s$ and for $D\subset W^{M_0R}_{\theta_1,\theta_2}$, define
\[
	\widetilde{\tau}^{i}_{D}=\inf\{t\ge 0, |\widetilde{B}^R_i(t)|\in D/R \}.
\]

\begin{lemma}\label{lem:TV_lem} Fix $\ep>0$ and $T_0>1$, and let $M_0=T_0^{1/2+\ep}$. 
Then for every $R>0$ sufficiently large and every $x_1,x_2\in \partial W^R_{\theta_1,\theta_2}$
\[
	\bigg|\sum_{y\in \overline{W}_{\theta_1,\theta_2}^{M_0R}/R}\big|\mathbf{P}^{x_1,x_2}_R(B^R_1(T_0)=y)-\mathbf{P}^{x_1,x_2}_R(B^R_2(T_0)=y) \big| - d_{\mathrm{TV}}(\widetilde{B}^R_1(T_0),\widetilde{B}^R_2(T_0))\bigg|\leq \frac{16C_\varepsilon}{T_0}
\]
\end{lemma}

\begin{proof}
	Fix $x_1,x_2\in\partial W^R_{\theta_1,\theta_2}$ and note that for $i\in \{1,2\}$
\[
\begin{aligned}
	\mathbf{P}^{x_1,x_2}_R\big(\tau^i_{\partial W^{M_0R/2}_{\theta_1,\theta_2}}\leq T_0\big) &\leq \mathbf{P}^{x_1,x_2}_R\big(N_i(T_0)\geq 4T_0 R^2\big) + \mathbf{P}^{x_1,x_2}_R\big(N_i(\tau^i_{\partial W^{M_0R/2}_{\theta_1,\theta_2}})\leq 4T_0R^2\big)\\
	&\leq e^{-2(\log(4)-1)T_0R^2} + \tbP^{x_i}_{\theta_1,\theta_2}(\tau_{\partial W^{M_0R/2}_{\theta_1,\theta_2}}\leq 4T_0R^2)\leq \frac{4C_\ep}{T_0}\,,
\end{aligned}
\]
where in the last step we used Lemma \ref{lemma_continuous_to_discrete}\eqref{continuous_to_discrete_2}. A similar argument shows that 
for $i\in \{1,2\}$
\[
	\mathbf{P}^{x_1,x_2}_R\Big(\tau^i_{\partial W^{M_0R/2}_{\theta_1,\theta_2}}\leq T_0\Big)\leq \frac{4C_\ep}{T_0}\,.
\]

Next, notice that the laws of $(B_i^R(t \wedge \tau^i_{\partial W^{M_0R/2}_{\theta_1,\theta_2}}))_{t\geq 0}$ and $(\widetilde{B}_i^R(t\wedge \widetilde{\tau}^i_{\partial W^{M_0R/2}_{\theta_1,\theta_2}}))_{t\geq 0}$ are equal and therefore
\[
	\mathbf{P}^{x_1,x_2}_R\big(B^R_i(T_0)=y,~\tau^i_{\partial W^{M_0R/2}_{\theta_1,\theta_2}}> T_0\big)
	=\widetilde{\mathbf{P}}^{x_1,x_2}_R\big(\widetilde{B}^R_i(T_0)=y,~\widetilde{\tau}^i_{\partial W^{M_0R/2}_{\theta_1,\theta_2}}> T_0\big).
\]

Combining all of the above, together with the fact that 
\[
	d_{\mathrm{TV}}(\widetilde{B}^R_1(T_0),\widetilde{B}^R_2(T_0))=\sum_{y\in W^{M_0R}_{\theta_1,\theta_2}/R }\big|\widetilde{\mathbf{P}}^{x_1,x_2}_R(\widetilde{B}^R_1(T_0)=y)-\widetilde{\mathbf{P}}^{x_1,x_2}_R(\widetilde{B}^R_2(T_0)=y) \big| 
\]
yields the result. 
\end{proof}

\begin{proof}[Proof of Proposition \ref{proposition_2}]
Combining the estimation for the sum over $y\in y\in (W_{\theta_1,\theta_2}\setminus \overline{W}_{\theta_1,\theta_2}^{M_0R})/R$ together with Lemma \ref{lem:TV_lem} implies that for every $x_1,x_2\in\partial W^R_{\theta_1,\theta_2}$ 
\[
	|\mathbf{P}^{x_1,x_2}_R(\CT>T_0)-d_{\mathrm{TV}}(\widetilde{B}^R_1(T_0),\widetilde{B}^R_2(T_0))|\leq \frac {20C_\ep}{T_0}\,,
\]
where $(\widetilde{B}^R_1(T_0)),\widetilde{B}^R_2(T_0)))$ is distributed according to the coupling $\widetilde{\tbP}^{x_1,x_2}_R$. Therefore, it suffices to show that
\[
	\lim_{T_0\to\infty}\lim_{R\to\infty}\sup_{x_1,x_2\in \partial W^R_{\theta_1,\theta_2}} d_{\mathrm{TV}}(\widetilde{B}^R_1(T_0),\widetilde{B}^R_2(T_0))=0\,.
\]

Recall that $(\widetilde{B}^R_1(t))_{t\geq 0}$ and $(\widetilde{B}^R_2(t))_{t\geq 0}$ are continuous time, simple random walks on $W^{M_0R}_{\theta_1,\theta_2}/R$ with law $\widetilde{\tbP}^{x_1,x_2}_R$, and in particular that they start in $x_1/R$ and $x_2/R$ respectively. 

We finish the proof using one last rescaling. For $i\in \{1,2\}$ and $R>0$, define
\[
	\widehat{B}^R_i(t)= \frac{1}{M_0}\widetilde{B}^R_1(M_0^2t),\qquad \forall t\geq 0.
\]

One can see that $(\widehat{B}^R_i(t))_{t\geq 0}$ for $i\in \{1,2\}$ are continuous time, simple random walks $W^{M_0R}_{\theta_1,\theta_2}/(M_0R)$ with constant jump rate $2(M_0R)^2$. In addition, note that $\overline{W}^{M_0R}_{\theta_1,\theta_2}/(M_0R)$ is also the intersection of the rescaled lattice $(M_0R)^{-1}\BZ^2$ and the continuous wedge $\CD=\CW_{\theta_1,\theta_2}\cap \{\|y\|_2<1\}$. Thus for any $y\in W_{\theta_1,\theta_2}/R$ such that $\|y\|_2<M_0$,
\begin{equation}\label{total_variation_11}
	\mathbf{P}^{x_1,x_2}_R\big(\widetilde{B}^R_i(T_0)=y\big)=\mathbf{P}^{x_1,x_2}_R\bigg(\widehat{B}^R_i\bigg(\frac{T_0}{M_0^2}\bigg)=\frac{y}{M_0}\bigg).
\end{equation}

Let $\widehat{T}_0=T_0/M_0^2 = T_0^{-2\ep}$. Then 
\begin{equation}\label{total_variation_12}
\begin{aligned}
d_{\mathrm{TV}}(\widetilde{B}^R_1(T_0),\widetilde{B}^R_2(T_0))&
=\sum_{\substack{y\in W_{\theta_1,\theta_2}/R \\  \|y\|_2<M_0}}\Big|\mathbf{P}^{x_1,x_2}_R\big(\widetilde{B}^R_1(T_0)=y\big)-\mathbf{P}^{x_1,x_2}_R\big(\widetilde{B}^R_2(T_0)=y\big) \Big|\\
&=\sum_{z\in W^{M_0R}_{\theta_1,\theta_2}/M_0R}\big|\mathbf{P}^{x_1,x_2}_R(\widehat{B}^R_1(\widehat{T}_0)=z)-\mathbf{P}^{x_1,x_2}_R(\widehat{B}^R_2(\widehat{T}_0)=z) \big|.
\end{aligned}
\end{equation}

We are now ready to use the aformentioned known bound on the mixing time for continuous-time random walks on bounded domains in $\BR^2$. Note that $\CD$ is a bounded, Lipschitz domain in $\BR^2$, and therefore, by (2.8) and Theorem 2.11 in \cite{Fan_2014}, for $\ep>0$ sufficiently small, there exists a constant, $C'\in (0,\infty)$ which are independent of $R$ and $z$, such that for all sufficiently large $R$ and $z\in W^{M_0R}_{\theta_1,\theta_2}/M_0R$
\begin{equation}\label{holder}
	(M_0R)^2\big|\mathbf{P}^{x_1,x_2}_R(\widehat{B}^R_1(\widehat{T}_0)=z)-\mathbf{P}^{x_1,x_2}_R(\widehat{B}^R_2(\widehat{T}_0)=z)\big|\le C'\frac{|\widehat{B}^R_1(0)-\widehat{B}^R_2(0)|^{8\ep}}{\widehat{ T}_0^{1+4\ep}}. 
\end{equation}
Recalling that $B^R_i(0)=x_i/(M_0R)$ for $i\in \{1,2\}$ and that $x_1,x_2\in \partial W^{R}_{\theta_1,\theta_2}$, we conclude that 
$|\widehat{B}^R_1(0)-\widehat{B}^R_2(0)|\le C M_0^{-1} = CT_0^{-1/2-\ep}$, and therefore
\[
	(M_0R)^2\big|\mathbf{P}^{x_1,x_2}_R(\widehat{B}^R_1(\widehat{T}_0)=z)-\mathbf{P}^{x_1,x_2}_R(\widehat{B}^R_2(\widehat{T}_0)=z)\big| \leq \frac{C'}{T_0^{2\ep}}
\]
Noting that $\mathrm{card}(W^{M_0R}_{\theta_1,\theta_2}/M_0R)=O((M_0R)^2)$, and that the bound is uniform in $x_1,x_2\in \partial W^R_{\theta_1,\theta_2}$, the proof of Proposition \ref{proposition_2} (and thus of Theorem \ref{theorem:harm} as well) is complete.
\end{proof}

\bibliography{career}
\bibliographystyle{alpha}

\end{document}

%% file: maxhitset.tex
\begin{tikzpicture}[scale=0.7]
\tikzstyle{redcirc}=[circle,
draw=black,fill=myred,thin,inner sep=0pt,minimum size=2mm]
\tikzstyle{bluecirc}=[circle,
draw=black,fill=blue,thin,inner sep=0pt,minimum size=2mm]

\draw [black,thin] (0,0) to (9.40,3.42);
\draw [red,thick] (0,0) to (9.40,-3.42);

\draw (10,0) arc (0:-20:10) ;
\draw (10,0) arc (0:20:10) ;

\def\Radius{2}
\def\radius{0}
\begin{scope}[even odd rule]
  \clip[rotate=-20.0] (0,0) -- (0:\Radius) arc (0:40:\Radius) --cycle;
  \fill[red] 
    circle[radius=\Radius]
    circle[radius=\radius]
  ;
\end{scope}

\draw [blue,thick] plot [smooth, tension=1.2] coordinates {(1.97,-0.3472) (3,-0.2) (4,-0.6) (6,0) (9.66,-2.59)};
\draw [black,thick] plot [smooth, tension=1.2] coordinates {(1.97,0.3472) (3,-0.2) (4,-1)(3.133,-1.14)};


\node (v3) at (1.5,1.1) {$r$};

\node (v4) at (9.3,3.9) {$L$};






\end{tikzpicture}

%% file: escapefar.tex
\begin{tikzpicture}[scale=0.5]
\tikzstyle{redcirc}=[circle,
draw=black,fill=myred,thin,inner sep=0pt,minimum size=2mm]
\tikzstyle{bluecirc}=[circle,
draw=black,fill=blue,thin,inner sep=0pt,minimum size=2mm]

\draw [black,thin] (0,0) to (10,0);
\draw [red,ultra thick] (0,0) to (5*1.732,-5);

\draw (2,0) arc (0:-30:2) ;
\draw (10,0) arc (0:-30:10) ;

\draw [blue,thick] plot [smooth, tension=2.5] coordinates {(1.97,-0.3472) (2,-0.2) (2.5,-0.6) (3,0) };
\draw [blue,thick] plot [smooth, tension=2.5] coordinates {(3,0) (4,-1.5)(5,-2) (5.5,-1.8) (6,0) };
\draw [blue,thick] plot [smooth, tension=2.5] coordinates {(6,0) (6.5,-3)(9.66,-2.59) };
\node (v1) at (1.97,-0.3472) [bluecirc] {};
\node (v1) at (9.66,-2.59) [bluecirc] {};

\node (v3) at (1.8,0.3) {$1$};
\node (v4) at (9.7,0.3) {$K$};
\node (v5) at (1.1,-0.3) {$\varphi$};




\end{tikzpicture}

%% file: escapefar_reflect.tex
\begin{tikzpicture}[scale=0.5]
\tikzstyle{redcirc}=[circle,
draw=black,fill=myred,thin,inner sep=0pt,minimum size=2mm]
\tikzstyle{bluecirc}=[circle,
draw=black,fill=blue,thin,inner sep=0pt,minimum size=2mm]

\draw [black,thin] (0,0) to (10,0);
\draw [red,ultra thick] (0,0) to (5*1.732,-5);
\draw [red,ultra thick] (0,0) to (5*1.732,5);

\draw (2,0) arc (0:-30:2) ;
\draw (10,0) arc (0:-30:10) ;
\draw (2,0) arc (0:30:2) ;
\draw (10,0) arc (0:30:10) ;

\draw [blue,thick] plot [smooth, tension=2.5] coordinates {(1.97,-0.3472) (2,-0.2) (2.5,-0.6) (3,0) };
\draw [blue,thick] plot [smooth, tension=2.5] coordinates {(3,0) (4,1.5)(5,2) (5.5,1.8) (6,0) };
\draw [blue,thick] plot [smooth, tension=2.5] coordinates {(6,0) (6.5,-3)(9.66,-2.59) };
\node (v1) at (1.97,-0.3472) [bluecirc] {};
\node (v2) at (9.66,-2.59) [bluecirc] {};

\node (v3) at (1.8,0.3) {$1$};
\node (v4) at (9.7,0.3) {$K$};

\node (v5) at (1.1,-0.3) {$\varphi$};
\node (v5) at (1.1,0.3) {$\varphi$};




\end{tikzpicture}

%% file: maxcoupling.tex
\begin{tikzpicture}[scale=0.6]
\tikzstyle{redcirc}=[circle,
draw=black,fill=myred,thin,inner sep=0pt,minimum size=2mm]
\tikzstyle{bluecirc}=[circle,
draw=black,fill=blue,thin,inner sep=0pt,minimum size=2mm]

\draw [black,ultra thick] (0,0) to (9.40,3.42);
\draw [black,ultra thick] (0,0) to (9.40,-3.42);

\draw (2,0) arc (0:-20:2) ;
\draw (10,0) arc (0:-20:10) ;
\draw (2,0) arc (0:20:2) ;
\draw (10,0) arc (0:20:10) ;
\draw (7,0) arc (0:20:7) ;
\draw (7,0) arc (0:-20:7) ;



\node (v4) at (9.3,3.9) {$M_0R$};

\node (v5) at (6.5,2.8) {$R$};

\node (v8) at (8,0.6) [bluecirc] {};
\node (v9) at (1.45,0.02)[red] {$A$};

\node (v6) at (6.97,-0.61) [bluecirc] {};
\node (v7) at (6.89,1.22) [bluecirc] {};
\draw [blue,thick] plot [smooth, tension=2.7] coordinates {(6.97,-0.61) (5,-0.8)(7,0.1) (8,0.6) (8.5,0.3) (9.9,1.4) };
\draw [blue,thick] plot [smooth, tension=2.5] coordinates {(6.89,1.22) (6,1.5)(5,0.5) (8,0.6) };


\draw [red,thick] plot [smooth, tension=2.7] coordinates {(0,0) (0.6,0.05)(0.8,-0.01)(1,-0.2) };
\draw [red,thick] plot [smooth, tension=2.7] coordinates {(0.8,-0.01)(0.9,0.1) (1.12,0.05) };
\draw [red,thick] plot [smooth, tension=2.7] coordinates {(0.6,0.05)(0.7,0.01)(0.8,0.03) (1.17,-0.07)};




\end{tikzpicture}